\documentclass{amsart}
\usepackage{amssymb,amsfonts}
\usepackage[all,cmtip]{xy}
\usepackage{mathrsfs}
\usepackage[scr=esstix,cal=boondox]{mathalfa}
\usepackage{hyperref}

\newtheorem{theorem}{Theorem}[section]
\newtheorem{corollary}[theorem]{Corollary}
\newtheorem{prop}[theorem]{Proposition}
\newtheorem{lemma}[theorem]{Lemma}
\newtheorem{obs}[theorem]{Observation}

\theoremstyle{definition}
\newtheorem{definition}[theorem]{Definition}
\newtheorem{remark}[theorem]{Remark}

\newtheorem{fact}[theorem]{Fact}

\newtheorem{notation}[theorem]{Notation}
\newtheorem{convention}[theorem]{Convention}
\newtheorem{conjecture}[theorem]{Conjecture}

\DeclareMathOperator{\aut}{Aut}
\DeclareMathOperator{\out}{Out}
\DeclareMathOperator{\inn}{Inn}
\DeclareMathOperator{\sym}{Sym}
\DeclareMathOperator{\alt}{Alt}

\DeclareMathOperator{\diag}{diag}
\DeclareMathOperator{\agl}{AGL}
\DeclareMathOperator{\gl}{AGL}

\DeclareMathOperator{\soc}{Soc}
\DeclareMathOperator{\supp}{supp}
\DeclareMathOperator{\asy}{asy}

\DeclareMathOperator{\nextt}{next}
\DeclareMathOperator{\solv}{solv}

\newcommand{\mn}{\medskip\noindent}

\newcommand{\normal}{\,\triangleleft\,}
\newcommand{\capp}{\,\cap\,}
\newcommand{\ie}{i.\,e.}
\newcommand{\eg}{e.\,g.}
\newcommand{\wrt}{w.\,r.\,t.\ } 
\newcommand{\vf}{\varphi}
\newcommand{\vfbar}{\overline{\varphi}}

\newcommand{\Deltabar}{\overline{\Delta}}
\newcommand{\nnn}{\mathbb N}
\newcommand{\zzz}{\mathbb Z}

\newcommand{\ccc}{\mathbb C}
\newcommand{\fff}{\mathbb F}
\newcommand{\calB}{\mathcal B}

\newcommand{\calD}{\mathcal D}

\newcommand{\calG}{\mathcal G}
\newcommand{\calH}{\mathcal H}
\newcommand{\calL}{\mathcal L}
\newcommand{\calQ}{\mathcal Q}
\newcommand{\calR}{\mathcal R}

\newcommand{\mfX}{\mathfrak X}

\newcommand{\inff}{^{(\infty)}}

\newcommand{\wt}{\widetilde}
\newcommand{\acts}{\curvearrowright}
\newcommand{\onto}{\twoheadrightarrow}

\newcommand{\semidirect}{\rtimes}

\def\beq{\begin{equation}}
\def\eeq{\end{equation}}

\def\bwhile{{\bf while}}

\def\bend{{\bf end}}

\def\breturn{{\bf return}}

\begin{document}

\title[Asymmetric coloring of locally finite graphs]%
      {Asymmetric coloring of locally finite graphs \\
        and profinite permutation groups:\\
        Tucker's Conjecture confirmed}
\author{L\'aszl\'o Babai}
\address{University of Chicago}

\dedicatory{To the memory of Jan Saxl}

\date{November 13, 2021}

\maketitle

\begin{abstract}
An \emph{asymmetric coloring} of a graph is a coloring of
its vertices that is not preserved by any non-identity
automorphism of the graph.  The \emph{motion} of a graph
is the minimal degree of its automorphism group, \ie,
the minimum number of elements that are moved (not fixed)
by any non-identity automorphism.  We confirm
Tom Tucker's ``Infinite Motion Conjecture'' that 
\emph{connected locally finite graphs with
infinite motion admit an asymmetric 2-coloring.}
We infer this from the more general result that the
\emph{inverse limit of an infinite sequence of finite permutation
groups with disjoint domains, viewed as a permutation
group on the union of those domains, admits an asymmetric
2-coloring.}
The proof is based on the study of the interaction
between epimorphisms of finite permutation groups and
the structure of the setwise stabilizers of subsets of 
their domains.
\end{abstract}

\section{Introduction}
A graph is \emph{locally finite} if every vertex has finite degree.
A graph is \emph{asymmetric} if it has no nontrivial
automorphisms.  A \emph{coloring} of the vertices of a graph
is asymmetric if no non-identity automorphisms of the graph
preserves the coloring.  This author established in 1977
that every regular tree of (finite or infinite) degree $\kappa \ge 2$
admits an \emph{asymmetric 2-coloring}~\cite{trees}.  While this
result is trivial for locally finite (and therefore countable)
trees (the point of~\cite{trees} was that it holds for arbitrary
infinite cardinals $\kappa$, and the difficulty begins at
strongly inaccessible cardinals), several interesting questions
about asymmetric colorings of locally finite graphs have
been asked and partly answered (see, \eg,
\cite{wong,trofimov,cuno14,watkins15,lehner16,lehner17,imrich-tucker17,
cuno14,pilsniak18,pilsniak19}).
A particularly intriguing
conjecture was formulated by Thomas W. Tucker in 2011~\cite{tucker}.
The \emph{degree} of a permutation is the number of elements it
moves.  The \emph{minimal degree} of a permutation group is the
minimum of the degrees if its non-identity elements.
The \emph{motion} of a graph is the minimal degree of its
automorphism group.  Tucker's \emph{Infinite Motion Conjecture}
states that \emph{every connected, locally finite graph with infinite
motion admits an asymmetric 2-coloring.}  A number of recent papers
obtained partial result on this conjecture, 
including~\cite{trofimov,cuno14,imrich-tucker17}.
Notably, Florian Lehner confirmed
the conjecture for graphs with intermediate ($\exp(O(\sqrt{n})$)
growth~\cite{lehner16}.  Lehner, Monika Pil{\'s}niak, and 
Marcin Stawiski confirmed the conjecture for graphs with 
maximum degree $\le 5$~\cite{pilsniak18}.
The same authors proved that if the maximum degree of the 
graph is $k$ then an asymmetric coloring with $O(\sqrt{k}\log k)$
colors exists~\cite{pilsniak19}.  

The main result of our paper confirms Tucker's conjecture
in full generality.  Along the way, we raise and partially
solve a number of questions regarding group-theoretic 
properties of colorings for finite permutation groups.

\begin{theorem}[Tucker's Infinite Motion Conjecture confirmed] 
\label{thm:main}
Let $X$ be a locally finite connected graph with infinite motion.
Then $X$ admits an asymmetric 2-coloring. 
\end{theorem}
In other words, the conclusion says that
the set of vertices has a subset,
not fixed setwise by any non-identity automorphism of $X$.

Given the fact that the automorphism group of a
connected locally finite \emph{rooted} graph is the
inverse limit of a sequence of finite permutation groups,
the proof boils down to the study of inverse systems of
epimorphisms among a sequence of finite permutation
groups.  The motion condition translates to
the \emph{disjointness} of the domains of the groups in the system.
Our result in this context is the following;
this is the main technical result of the paper.
\begin{theorem}[asymmetric coloring of inverse limit]  \label{thm:main1}
Let $\calG$ be the inverse limit of an
infinite sequence of finite
permutation groups with disjoint domains, viewed as
a permutation group acting on the union of those domains.
Then $\calG$ admits an asymmetric 2-coloring.
\end{theorem}
In other words, the conclusion says that
the domain of $\calG$ has a subset,
not fixed setwise by any non-identity element of $\calG$.

Our proof builds on a line of work on asymmetric 
colorings of finite primitive permutation groups,
started by David Gluck (1983)~\cite{gluck}
and Peter Cameron, Peter M. Neumann, and Jan Saxl (1984)~\cite{saxl},
a theory to which we contribute in this paper.
By a counting argument (see Prop.~\ref{prop:motionlemma}),
also used by Gluck, Cameron et al.
proved that in addition to the symmetric and alternating
groups, there are only a finite number of 
primitive groups that do not admit an asymmetric
2-coloring (see Theorem~\ref{thm:saxl}).
Seress classified the exceptions (1997)~\cite{seress97}.
Our proof depends on Seress's classification.

The result of~\cite{saxl} depends on the
Classification of Finite Simple Groups (CFSG)
through a result by Cameron~\cite{cameron81}.
Seress uses detailed explicit knowledge of CFSG.
While Cameron's result, used in~\cite{saxl},
has recently been given a remarkable elementary
combinatorial proof by Sun and Wilmes~\cite{sunwilmes}
(see Sec.~\ref{sec:cfsg-free}),
through Seress's work, our paper continues to depend
on CFSG.  We suspect, though, that our proof can be
modified to avoid dependence on CFSG.

The automorphism group of a 2-coloring is the same as the
setwise stablizer of a subset of the permutation domain.
Our method is that we approximate asymmetry by
gradually simplifying the structure of the groups
involved in the inverse limit, by 2-coloring the
underlying sets of an infinite, coinfinite subset
of the finite groups that participate in the inverse limit.
The hard part is to reduce all groups
to solvable groups.  To this end, we introduce the concept
of \emph{solvable coloring:} a coloring that is preserved
only by a solvable subgroup of the automorphism group.
This concept, intermediate on the way to asymmetry, may
deserve a systematic study.  Once reduced to solvable
groups, we reduce the groups to bounded derived length,
and finally we show how to reduce the derived length,
until, at derived length zero, the group vanishes and asymmetry occurs.

Now we state our key lemma, which may be of independent interest.
\begin{theorem}[Reducing simple image]   \label{thm:simple}
Let $G\le\sym(\Omega)$ be a finite permutation group and
$T$ a finite nonabelian simple group.  Let
$\vf: G\onto T$ be an epimorphism.  Then
there exists a subset $\Delta\subseteq\Omega$
such that $\vf(G_{\Delta}) < T$, where $G_{\Delta}$
denotes the setwise stabilizer of $\Delta$.
\end{theorem}

We shall need the following two additional results 
about colorings for finite permutation groups.  

\begin{theorem}[Bounded orbits for solvable groups] \label{thm:bded-orb}
There is a constant $C$ such that the following holds.
Let $G\le\sym(\Omega)$ be a solvable finite permutation group.
Then there exists a subset $\Delta\subseteq\Omega$
such that every orbit of $G_{\Delta}$ has length $\le C$.
\end{theorem}

We use this result only to infer that 
$G_\Delta$ has bounded derived length.

Along the way to proving Theorem~\ref{thm:bded-orb},
we show that every solvable permutation group
has an asymmetric 5-coloring (Lemma~\ref{lem:solv5color}).
(The bound 5 is tight, see Remark~\ref{rem:5colorstight}.)

The following observation will allow us to
move from bounded derived length to asymmetry.

\begin{prop}   \label{prop:derivedlength-intro}
Every nontrivial solvable permutation group
admits a 2-coloring that reduces its derived length.  
\end{prop}

\subsection{Structure of the paper}
We briefly review the author's view of the 
history of combinatorial symmetry breaking 
in Sec.~\ref{sec:history}.  Rudimentary
definitions follow in Sec.~\ref{sec:def1}.
We reduce Tucker's conjecture to the study of
epimorphisms of finite permutation groups in 
Sections~\ref{sec:inverse}--\ref{sec:approximation}.
This includes a review of inverse systems of group
homomorphisms in Sec.~\ref{sec:inverse}.
The program of gradual structural reduction of the groups
is formalized in Theorem~\ref{thm:program}.

After Sec.~\ref{sec:approximation}, all groups considered
are finite.  Further details of group theoretic notation
are reviewed in Sec.~\ref{sec:defs}.
The first phase of our program of structural reductions,
going from general to solvable groups,
occupies Sections~\ref{sec:solvable} to~\ref{sec:red-to-solvable}.
This includes the proof of Theorem~\ref{thm:simple}
for primitive groups in Sec.~\ref{sec:primitive}
and for general permutation groups in Sec.~\ref{sec:general2simple}.

We achieve bounded derived length in Section~\ref{sec:bded-derived-length}
by proving Theorem~\ref{thm:bded-orb}.
The reduction of the derived length (Prop.~\ref{prop:derivedlength-intro})
is the subject of Sec.~\ref{sec:reducing-derived-length}.
This completes the proof of the Infinite Motion Conjecture.

In Section~\ref{sec:effective} we state a finite version of
Theorem~\ref{thm:main1} (Theorem~\ref{thm:finite}) and
state two conjectures in search for more effective versions
of this result.

In Section~\ref{sec:cfsg-free} we address the question of CFSG-free
proofs and combinatorial generalizations, notably to the motion
of primitive coherent configurations.

The techniques introduced in this paper give rise
to a number of new questions and potential areas
of research; some of these are listed in the
concluding section of this paper, Sec.~\ref{sec:open}.

\section{A brief history of combinatrial symmetry breaking}
\label{sec:history}
The study of asymmetry dates back to the 1939 paper by 
Roberto Frucht~\cite{frucht39} which established
that every finite group is isomorphic to the automorphism
group of a graph.  Frucht introduced asymmetric ``gadgets''
to code colors and thereby eliminate unwanted automorphisms.
A decade later Frucht proved the same result for 3-regular 
graphs~\cite{frucht49}; as a tool, he constructed
an asymmetric 3-regular graph that today bears his name.
Analogues of Frucht's Theorem were found to hold for
many other classes of structures, notably including
Steiner Triple Systems~\cite{mendelsohn} and algebraic number
fields~\cite{fried}.  A theory was developed to construct
graphs whose automorphism group was a prescribed regular
permutation group (GRR theory \cite{hetzel, godsil, DRR, morris}).  
In all these cases, 
it is easy to construct structures on which the given group acts;
the problem is the elimination of unwanted symmetry by coding
asymmetry into the structure to get exactly the desired group of
automorphisms.
In the 1960s, this line of work was extended by the Prague school
of category theory to endomorphisms~\cite{pultr65}. A key aspect of
this generalization was the construction of \emph{rigid graphs},
\ie, graphs with no nontrivial 
endomorphisms~\cite{pultr65,vopenka65,lambek69}.

It was in this context that this author considered the
simplest possible asymmetric graphs 
in the infinite case, establishing that for any two
(finite or infinite) cardinal numbers $\kappa > \lambda \ge 2$
there exists an asymmetric tree with only these two
degrees~\cite{trees}.
The key auxiliary result was that a regular tree
of (finite or infinite) degree $\kappa \ge 2$ admits
an \emph{asymmetric 2-coloring}---the first result
on asymmetric colorings.

In 1983/84, the idea of asymmetric 2-colorings was introduced
in group theory in a pair of independent papers,
by Gluck~\cite{gluck} and by Cameron--Neumann--Saxl~\cite{saxl}.
Gluck found that all but a finite number of solvable primitive
groups admit an asymmetric 2-coloring, and gave the exact
list of exceptions (all of them of degree $\le 9$).
Cameron--Neumann--Saxl proved
that all but a finite number
of primitive permutation groups other than $A_n$ and $S_n$
admit an asymmetric 2-coloring.  Subsequently
Seress~\cite{seress97} classified all the exceptions;
Seress's paper is one of our key references.

The motivation of Gluck, Cameron et al., and Seress
was in classical questions of the theory of 
permutation groups and partly, in questions of computational
group theory through the closely related concept of \emph{bases}
of permutation groups (see Seress's monograph~\cite{seressbook}).
These papers discovered the relationship of asymmetric
colorings with the \emph{minimal degree} of the permutation
group, \ie, the minimum number of elements not fixed by a
non-identity element of the group.  This is a classical concept,
studied since the time of Jordan~\cite{jordan} and 
Bochert~\cite{bochert} in the 19th century 
(see~\cite[Sec. 3.3]{dixon-mortimer}).

A \emph{base} of a permutation group $G\le\sym(\Omega)$
is a subset $\Delta\subseteq\Omega$ such that the
pointwise stabilizer $G_{(\Delta)}$ is the identity.
This classical ``symmetry breaking'' concept, evidently 
closely related to the notion of asymmetric colorings,
gained interest in computational group theory through
the work of Charles Sims in the 1960s~\cite{sims}
(see~\cite{seressbook}).  The significance of this concept
to asymptotic group theory comes from the observation
that if $G\le S_n$ has a base of size $b$ then $|G|\le n^b$.
Bounding the orders of \emph{primitive} permutation groups
was a central question of 19th century group theory
(see \cite{praeger} for some of this history).

In 1979, this author found a graph theoretic method to bound the
base size of primitive permutation groups in the more general
context of \emph{primitive coherent configurations} (PCCs)--certain
highly regular colorings of the edges of the complete directed graph
(see Sec.~\ref{sec:PCC}).  Entirely omitting group theory,
this method nevertheless produced a nearly tight upper
bound on the base size ($O(\sqrt{n}\log n)$), and therefore
on the order, of any primitive group other than the
alternating and the symmetric groups, solving a then century-old
problem~\cite{annals}.  The key technical result was a
symmetry-breaking tool:
how many vertices can distinguish between a given pair of vertices,
see Theorem~\ref{thm:annals}.  As a byproduct, a lower bound on
the minimal degree of primitive groups and on the motion of
PCCs follows (Theorem~\ref{thm:annals} and
Obs.~\ref{obs:uni-motion}); the bound is tight
within a constant factor.

Significant progress over this result occurred in 2015 when
Xiaorui Sun and John Wilmes extended the result to a classification
of all primitive groups of minimum base size greater than essentially
$n^{1/3}$, while building a structure theory of
primitive coherent configurations along the way.
This remarkable result, previously
only known through CFSG~\cite{cameron81}, can be used to give
an elementary proof of the Cameron--Neumann--Saxl result mentioned
(see Sec.~\ref{sec:cfsg-free}).
It would presumably play a prominent role in a CFSG-free
proof of Tucker's Conjecture.  

The \emph{Graph Isomorphism problem} has been a chief producer
and consumer of cost-conscious symmetry breaking techniques.
Both~\cite{annals} and~\cite{sunwilmes} were partly motivated by
the complexity of this problem.
Many new details about this connection, in the context of
coherent configurations, emerged in~\cite{quasipoly},
where a quasipolynomial-time algorithm for testing
graph isomorphism is described.  Comments on 
measuring the cost appear in Sec.~\ref{sec:open}.

\mn
{\bf Terminology.} \quad
A much-cited 1996 paper by 
Michael Albertson and Karen Collins \cite{collins}
introduced asymmetric colorings under the name
``distinguishing coloring.''  Although this terminology
has found a large following, I find it unfortunate
for multiple reasons and will continue to use the term 
``distinguish'' in more natural meanings, including in this paper
(see right before Theorem~\ref{thm:annals}, and
Problem~\ref{pr:cameron-schemes}
in Sec.~\ref{sec:open}).  The term ``asymmetric coloring'' was
introduced in~\cite{trees} (1977).

The term ``minimal degree'' of a permutation group has been in
use at least since Wielandt's 1964 book~\cite{wielandt}, and
has since been used in a large body of literature in the theory
of permutation groups.  However, this term is ill-suited for 
applications to the automorphism group of a graph, where it
could be confused with the minimum degree of the graph---an
unrelated concept.  Therefore I have adopted the term
``motion'' of a graph or other structure, meaning the minimal
degree of its automorphism group, as suggested in~\cite{russell}.

\section{Definitions, notation: group actions and coloring}
\label{sec:def1}
Structural group theoretic definitions and notation
will be reviewed in Section~\ref{sec:defs}.  The beautiful
monograph~\cite{dixon-mortimer} covers most of the
group theory we need.  In this section we only deal
with the most rudimentary concepts.

\begin{definition}[Coloring]
Let $\Sigma$ be an ordered set; we refer to the elements
of $\Sigma$ as ``colors.''  A $\Sigma$-coloring of a
set $\Omega$ is a function $\gamma: \Omega\to \Sigma$.
If $|\Sigma|=k$, we speak of a $k$-coloring and
often use $\Sigma :=[k]=\{1,2,\dots, k\}$ as the set of colors.
We identify the subsets of $\Omega$
with 2-colorings, using a fixed ordered pair of colors.
\end{definition}

\begin{notation}[Symmetric and alternating groups]
For a set $\Omega$, we write $\sym(\Omega)$ and $\alt(\Omega)$
for the symmetric and the alteranting group, resp., on $\Omega$.
We also write $S_n$ for the generic symmetric group of
degree $n$ and $A_n$ for the alternating group of degree $n$.
\end{notation}

\begin{definition}[Group action]  \label{def:action}
Let $G$ be a group, $\Omega$ a set.  A \emph{$G$-action}
on $\Omega$, denoted $G\acts\Omega$,
is a homomorphism $\vf: G\to\sym(\Omega)$.
We refer to $\Omega$ as the \emph{permutation domain.}
The \emph{image} of the action $G\stackrel{\vf}{\acts}\Omega$
is the group $\vf(G)\le\sym(\Omega)$.
For $\pi\in G$ and $x\in\Omega$ we write $\pi(x)$ to
denote $(\vf(\pi))(x)$ if the action $\vf$ is clear
from the context.
\end{definition}

\begin{definition}[Coloring for a group action]
Let $\vf: G\acts\Omega$ be a group action.
By a coloring for $\vf$ we mean a coloring of $\Omega$.
For a coloring $\gamma:\Omega\to\Sigma$ and $\pi\in G$
we write $\pi(\gamma)$ to denote the coloring
$(\pi(\gamma))(x)=\gamma(\pi^{-1}(x))$ $(x\in\Omega)$.
\end{definition}

\begin{definition}[Stabilizers]
Let $\vf: G\acts\Omega$ be a group action.
Let $x\in\Omega$.  The \emph{stabilizer} of $x\in\Omega$
is the subgroup $G_x=\{\pi\in G\mid \pi(x)=x\}$.
For $\Delta\subseteq\Omega$ we write
$G_\Delta=\{\pi\in G\mid \pi(\Delta)=\Delta\}$
for the \emph{setwise stabilizer} of $\Delta$
and $G_{(\Delta)}=\bigcap_{x\in\Delta} G_x$
for the \emph{pointwise stabilizer} of $\Delta$.
The \emph{stabilizer of the coloring} $\gamma$ is the
subgroup $G_\gamma = \{\pi\in G\mid \pi(\gamma)=\gamma\}$.
We refer to $G_\gamma$ as the group of \emph{$G$-automorphisms
of the coloring $\gamma$.} In other words, the $G$-automorphisms 
of the coloring $\gamma$ are those elements of $G$ that 
preserve $\gamma$.
\end{definition}
\begin{obs}[Subsets vs. 2-colorings]   \label{obs:set-stab}
If $\Sigma=\{a,b\}$ where $a < b$ 
and $\gamma:\Omega\to\Sigma$ is a 2-coloring
then $G_\gamma=G_\Delta$ where $\Delta=\gamma^{-1}(b)$.
    \hfill $\Box$ 
\end{obs}

\begin{definition}[Permutation group]
A \emph{permutation group} acting on $\Omega$ is a subgroup of 
$\sym(\Omega)$.  We view permutation groups
as faithful actions.  A \emph{coloring} for $G$
is a coloring of $\Omega$.  (This means a coloring for this
faithful action).
\end{definition}

\begin{definition}[Group theoretic properties of colorings] \label{def:asy}
Let $G\acts\Omega$ be an action.  We say that a coloring $\gamma$ 
for $G$ is \emph{asymmetric} if $G_\gamma=1$.
We also say that such a coloring of $\Omega$ is $G$-asymmetric.
Analogously we speak of \emph{$G$-asymmetric subsets}
of the domain.   We call the minimum number of colors
required by an asymmetric coloring for $G$ the
\emph{asymmetric coloring number} of $G$.
We say that $\gamma$ is a \emph{solvable coloring} if $G_\gamma$
is a solvable group.  We also express this circumstance by saying
that the coloring $\gamma$ \emph{results in a solvable group}.
We call the minimum number of colors required by a solvable coloring
the \emph{solvable coloring number} of $G$.

We may analogously ascribe other group theoretic properties
to $\gamma$.  For example, $\gamma$ can result in a 2-group
(\ie, $G_{\gamma}$ is a 2-group) or in a solvable group
with derived length $\le 3$, etc.

We define the corresponding concepts
for subsets of the domain (asymmetric subsets,
solvable subsets, subsets resulting in a 2-group, etc.) 
via the correspondence to 2-colorings (Obs.~\ref{obs:set-stab}).
We say that $G$ \emph{admits} a coloring with certain property
(\eg, a solvable $k$-coloring) if there is a coloring 
of the domain with the given property.
\end{definition}
Note that an asymmetric coloring exists if and only
if the action of $G$ is faithful.

\begin{definition}[Support, minimal degree]
Let $G\le\sym(\Omega)$ be a permutation group.
The \emph{support} of $\sigma\in G$ is the set
$\supp(\sigma):=\{x\in\Omega\mid \sigma(x)\neq x\}$.
The \emph{minimal degree} of $G$ is
$\mu(G):=\min_{\sigma\in G\setminus\{1\}}|\supp(\sigma)|$\,,
the minimum number of elements moved (not fixed) by 
any non-identity element of $G$.  If $|G|=1$
then its minimal degree is ``super-infinity,'' denoted $\infty$ 
and thought of as being greater than any cardinal number.
\end{definition}

For instance, for $n\ge 3$, $\mu(S_n)=2$ and $\mu(A_n)=3$.

\begin{definition}[Coloring of structures]
Let $\mfX=(\Omega,\calR)$ be a structure (such as a graph).
We say that a coloring $\gamma:\Omega\to\Sigma$ of the
underlying set (set of vertices) has a certain property
\wrt $\mfX$ (such as being asymmetric or solvable)
if it has the corresponding property \wrt $\aut(\mfX)$.
\end{definition}
\begin{definition}[Motion]
The \emph{motion} of a structure $\mfX$ is the minimal degree
of $\aut(\mfX)$.  
\end{definition}
The term ``motion'' in this meaning was introduced
in \cite{russell}.

\mn
Let $G\le \sym(\Omega)$ be a permutation group.  
We say that $G$ is a \emph{finite permutation group}
if its domain $\Omega$ is finite.

\section{Inverse systems of epimorphisms}
\label{sec:inverse}

Inverse systems can be indexed by an arbitrary poset.  In this paper,
we only consider inverse systems of infinite \emph{sequences} of groups,
so the index set is $\nnn=\{0,1,\dots\}$, the set of natural numbers,
and the infinite subsets of $\nnn$, under the natural ordering.
\begin{convention}
Throughout this paper, the letters $I$ and $J$ will denote infinite
subsets of $\nnn$, with the natural ordering.
\end{convention}
\begin{definition}[Inverse system of homomorphisms of finite groups]
Let $(G_i\mid i\in I)$ be an infinite sequence of finite groups.
For $i \le j$ $(i,j\in I)$ let $\vf_{i,j} : G_j \to G_i$ be a 
homomorphism such that for $i \le j\le k$ $(i,j,k\in I)$, the
\emph{compatibility condition} $\vf_{i,j}\vf_{j,k} = \vf_{i,k}$ holds. 
The homomorphism $\vf_{ii}$ is the identity on $G_i$.
The $\vf_{i,j}$ are called the \emph{transition homomorphisms}.
The groups $G_i$, together with the transition homomorphisms,
form an \emph{inverse system}, denoted $(G_i, \vf_{i,j})_I$,
or simply $(G_i, \vf_{i,j})$ if the index set $I$ is clear from
the context.  We say that the system is \emph{epimorphic} if all 
transition homomorphisms are surjective.
\end{definition}
\begin{definition}[Inverse limit, strands]  \label{def:strand}
Let us consider a sequence $(g_i \mid i\in I)$ of group
elements, $g_i\in G_i$.  
Given the inverse system above, we call such a 
sequence a \emph{strand} if for all $i \le j$ we have
$g_i = \vf_{i,j}(g_j)$.  The inverse limit 
$\calG =\varprojlim G_i$ is the subgroup of the
direct product of the $G_i$ consisting of the strands. 
Viewing the $G_i$ as discrete groups and endowing their
direct product with the product topology, $\calG$ is
a closed subgroup of the direct product (a profinite group).
\end{definition}

\begin{fact}   \label{fact:aut-is-limit}
Let $X$ be a connected, locally finite, infinite 
graph, with possibly colored vertices.  Let $x_0$ be a vertex.
Then $\aut(X)_{x_0}$ (the stabilizer of $x_0$ in $\aut(X)$)
is the inverse limit of a sequence of finite permutation groups.
\end{fact}
\begin{proof}[Idea of proof]
The finite permutation groups in question are the
automorphism groups of the balls of each radius about $x_0$.
We can alternatively also take the restrictions of $\aut(X)$  
to those balls; this would correspond to the epimorphic
reduction discussed below.
\end{proof}

\begin{fact}
Let $(G_i, \vf_{i,j})_I$ be an \emph{epimorphic}
inverse system with limit $\calG =\varprojlim G_i$.
Let $\pi_i: \calG \to G_i$ denote the projection to the $i$-th
coordinate.  Then each $\pi_i$ is an epimorphism.
\hfill $\Box$ 
\end{fact}

\begin{fact}[Epimorphic reduction]  \label{fact:epiredux}
Let $(G_i, \vf_{i,j})_I$ be an inverse system of finite groups 
with inverse limit $\calG = \varprojlim G_i$.
Then there exist subgroups $H_i\le G_i$ such that
$(H_i, (\vf_{i,j})_{|H_j})_I$ is an \emph{epimorphic}
inverse system such that $\calG = \varprojlim H_i$.  We call
this system the \emph{epimorphic reduction} of the system
$(G_i, \vf_{i,j})_I$.  The epimorphic reduction is unique.
\end{fact}
\begin{proof}
Let $\pi_i: \calG \to G_i$ denote the projection to the $i$-th
coordinate.  Let $H_i = \pi_i(\calG)$.
Then the system consisting of the groups $H_i$
with the transition maps $(\vf_{i,j})_{|H_j}$
is clearly the unique inverse system satisfying
the conditions.
\end{proof}

\begin{prop}[Sublimit]
Let $J\subseteq I\subseteq\nnn$  
and let $(G_i,\vf_{i,j})_I$ be an inverse system
of finite groups.  Let $\calG=\varprojlim_{i\in I} G_i$
and let $\calH=\varprojlim_{j\in J} G_j$.  Then the projection
$\calG\onto\calH$ is an isomorphism. \hfill $\Box$ 
\end{prop}

\begin{definition}[$\Lambda$-reduction]
Let $(G_i,\vf_{i,j})_I$ 
be an inverse system as above,
with inverse limit $\calG$, and
let $\Lambda=(L_i \mid i\in I)$ where $L_i\le G_i$.  
The \emph{$\Lambda$-reduction} $\calL$ of $\calG$ consists
of those strands $(g_i : i\in I)$ 
where $g_i\in L_i$ for all $i$.
\end{definition}
\begin{fact}
$\calL$ is a closed subgroup of $\calG$ and 
$\calL = \varprojlim_{i\in I} L_i$.
\end{fact}
Let us now consider an inverse system $(G_i,\vf_{i,j})_I$
where each $G_i$ is a finite permutation group, 
$G_i\le \sym(\Omega_i)$.  
Critically, we assume the $\Omega_i$ are \emph{disjoint}
(an assumption that is patently \emph{false} in the proof of
Fact~\ref{fact:aut-is-limit}).

\begin{definition}[Action of inverse limit of permutation groups
    with disjoint domains]  \label{def:limperm}
Let $(G_i,\vf_{i,j})_I$ be an inverse system of finite
permutation groups $G_i\le \sym(\Omega_i)$. 
Assume the $\Omega_i$ are \emph{disjoint}.
Let $\Omega = \bigsqcup_{i\in I}\Omega_i$.
We view the direct product $\prod_{i\in I} G_i$
as a permutation group acting on $\Omega$ (coordinatewise).
As a consequence, the inverse limit $\calG = \varprojlim_{i\in I} G_i$
is also a permutation group acting on $\Omega$.
\end{definition}

\begin{obs}[Coloring of inverse limit of
    permutation groups with disjoint domains]
Under the assumptions of Def.~\ref{def:limperm},
let $\gamma$ be a coloring of the set
$\Omega=\bigsqcup_{i\in I}\Omega_i$\,.
In accordance with our conventions,
we say that $\gamma$ is a \emph{coloring for the inverse limit}
$\calG = \varprojlim G_i$.
Let $\gamma_i$ be the coloring $\gamma$ restricted to $\Omega_i$.
Let $L_i=(G_i)_{\gamma_i}$ and set $\Lambda = (L_i : i\in I)$.
Let $\calL\le\calG$ denote the $\Lambda$-reduction of $\calG$.
Then $\calL = \calG_\gamma$.  In particular,
$\gamma$ is an \emph{asymmetric coloring} for $\calG$ if and only if
$\calL = 1$.      \hfill $\Box$ 
\end{obs}
  
The Infinite Motion Conjecture will easily follow from the 
following result, previously stated as Theorem~\ref{thm:main1}.
\begin{theorem}[main technical result]  \label{thm:main2}
Let $\calG$ be the inverse limit of an infinite sequence of finite
permutation groups with disjoint domains.
Then $\calG$ admits an asymmetric 2-coloring.
\end{theorem}

The disjointness condition illuminates the role of the
\emph{infinite motion assumption} in Tucker's Conjecture.
The latter will permit us to replace the \emph{balls} 
in the proof of Fact~\ref{fact:aut-is-limit}
by the \emph{spheres}, which are disjoint 
(see Lemma~\ref{lem:invertible} below).

\section{Reduction of Tucker's Conjecture to Theorem~\ref{thm:main1}}
\label{sec:tucker-reduction}
First we reduce the problem to the case of rooted graphs,
where a designated vertex (the ``root'') is fixed by
all automorphisms. 

\begin{notation}[Spheres, balls]
Let $\rho$ denote the distance metric in the connected graph $X$
with vertex set $V$.
For a vertex $v\in V$, let $S_d(v)=\{w\in V\mid \rho(v,w)=d\}$
denote the \emph{sphere} of radius $d$ about $v$ and
$B_d(v)=\{w\in V\mid \rho(v,w)\le d\}$
the \emph{ball} of radius $d$ about $v$.
\end{notation}
\begin{definition}[Twins]
Let us call the vertices $u\neq v$ of the graph $X$
\emph{twins} if the transposition $(u,v)$ is an automorphism of $X$.
We say that the graph $X$ \emph{twin-free} if there are
no twins in $X$.
\end{definition}
\begin{obs}
If $X$ has infinite motion then it is twin-free.
    \hfill $\Box$
\end{obs}

\begin{definition}[Special subset]
Let $X$ be a connected locally finite graph with vertex set $V$.
Let $x_0\in V$.  We call a subset $\Delta\subseteq V$ \emph{special}
(with respect to $x_0$)
if $\Delta\capp B_1(x_0)=\emptyset$ and 
$S_{2d}(x_0)\subseteq\Delta$ for all $d\ge 1$.
\end{definition}

\begin{lemma}[Designated root]  \label{lem:root}
Let $X$ be a connected, twin-free graph with vertex set $V$.
Let $x_0\in V$ and let $\Delta$ be a special subset of $V$
(with respect to $x_0$).
Then $\aut(X)_\Delta$ fixes $x_0$.
\end{lemma}
\begin{proof}
Let $Z=\{v\in V\mid B_1(v)\capp\Delta =\emptyset\}$.
Then $x_0\in Z\subseteq B_1(x_0)$.  Moreover, for each
$z\in Z$ we have $B_1(z)\subseteq B_1(x_0)$.
Since $X$ is twin-free, we in fact have $B_1(z)\subset B_1(x_0)$.
Consequently $x_0$ is the unique vertex in $Z$ adjacent
to all other vertices in $Z$.
\end{proof}

Using Theorem~\ref{thm:main1} we shall show the following.
\begin{theorem}  \label{thm:rooted}
Let $X$ be a locally finite connected
graph with infinite motion and let $x_0$ be a vertex.
Then $X$ admits an asymmetric special subset \wrt $x_0$.
\end{theorem}
Note that here we do not make it an assumption that 
all automorphisms fix $x_0$.  The automorphisms that
fix a special subset will automatically fix $x_0$
by Lemma~\ref{lem:root}.

For a locally finite connected rooted graph $X=(V,E,x_0)$
we use the following notation.

Let $\calG=\aut(X)$.  Here we made $x_0$ a constant in
the language of the structure $X$, so this group fixes $x_0$ by
definition.
Consequently, $\calG$ also fixes each sphere about $x_0$ (setwise).
Let $B_i=B_i(x_0)$ and $S_i=S_i(x_0)$.
Let $H_i$ denote the restriction of $\calG$ to 
$B_i$ and $G_i$ the restriction of $\calG$ to $S_i$
(so $H_i\le\sym(B_i)$ and $G_i\le\sym(S_i)$).
For $i\in I$ let $\pi_i: H_i \onto G_i$ denote
the restriction from $B_i$ to $S_i$ (an epimorphism)
and for $i\le j$ let $\psi_{i,j}: H_j\onto H_i$ denote the
restriction from $B_j$ to $B_i$ (again, an epimorphism).

\begin{lemma}   \label{lem:invertible}
Let $X$ be a locally finite connected
rooted graph with infinite motion.
Then each restriction epimorphism $\pi_i$
(from the ball $B_i$ to the sphere $S_i$)
is an isomorphism.
\end{lemma}
\begin{proof}
Let $\sigma\in\ker\pi_i$.  So $\sigma$ acts on $B_i$ and
fixes $S_i$ pointwise.  Let $\widehat \sigma$ denote 
the extension of $\sigma$ to $V$ obtained by fixing
all vertices in $V\setminus B_i$.  So $\widehat \sigma$
is an automorphism of $X$ that fixes all vertices
outside the finite set $B_i$ and therefore,
by the infinite motion assumption, it fixes 
all vertices in $B_i$ as well, so $\sigma=1$.
\end{proof}

\begin{proof}[Proof of Theorem~\ref{thm:rooted}
from Theorem~\ref{thm:main1}]
By Lemma~\ref{lem:invertible}, we can define,
for $i \le j$, the $G_j\onto G_i$ epimorphism
$\vf_{i,j}=\pi_j\psi_{i,j}\pi_i^{-1}$.
Let $\calG$ denote the inverse limit of the
system $(G_i,\vf_{i,j})_{2\nnn +3}$.
Recall that the domain of $G_i$ is the sphere $S_i$,
so the domains of these permutation groups
are disjoint.

Let $\widehat\calG$ denote the restriction of
$\calG$ to the set $A=\bigsqcup\{S_i \mid i\in 2\nnn +3\}$.
By Theorem~\ref{thm:main1}, the group $\widehat\calG$ admits
an asymmetric subset $\Gamma$.  Let now 
$\Delta = \Gamma \cup \bigcup \{S_j\mid j\in 2\nnn+2\}$.
Then $\Delta$ is a special subset and $(\calG)_{\Delta}$
fixes $A$ pointwise.  But then the epimorphisms $\vf_{i,j}$
ensure that all of $V$ is fixed pointwise.
\end{proof}

\section{Approximation process}
\label{sec:approximation}

We say that a class $\calQ$ of groups is \emph{HS-closed}
if $\calQ$ is closed under subgroups and homomorphic
images.  (This includes being closed under isomorphisms.)

\begin{obs}
Let $(G_i, \vf_{i,j})_I$ be an \emph{epimorphic} system of
finite groups.  If $\calQ$ is an HS-closed class and infinitely
many of the $G_i$ belong to $\calQ$ then all of them belong
to $\calQ$.     \hfill $\Box$ 
\end{obs} 

\begin{definition}[Color-reduction between classes of permutation groups]
Let $\calQ$ and $\calR$ be HS-closed classes of finite permutation groups.  
We say that $\calQ$ is \emph{$k$-color-reducible} to $\calR$
if the following holds for all pairs of permutation groups
$G,H\in\calQ$.  If $\vf:H\onto G$ is an epimorphism and
$G\notin \calR$ then there exists a $k$-coloring $\gamma$
for $H$ such that $\vf(H_\gamma) < G$.
\end{definition}

\begin{lemma}[Color-reduction of inverse limits]  \label{lem:reduction}
Let $\calQ$ and $\calR$ be HS-closed classes of finite groups.  
Assume $\calQ$ is $2$-color-reducible to $\calR$.  
Let $(G_i,\vf_{i,j})_I$ be an epimorphic inverse system 
of finite permutation groups $G_i\in \calQ$; let $\Omega_i$ denote
the permutation domain of $G_i$.  Assume the $\Omega_i$ are disjoint.
Let $\calG = \varprojlim G_i$.  Let $J$ be an infinite subset of $I$.
Then there exists a subset $\Delta\subseteq \bigcup_{i\in J} \Omega_i$
such that for all $i\in I$ we have $\pi_i(\calG_{\Delta})\in\calR$.
\end{lemma}
\begin{proof}
For a subset $K\subseteq I$ let us use the notation
$\Omega(K):=\bigcup_{i\in K} \Omega_i$.
For $i\in I$, in increasing order, we shall inductively designate
a finite subset $J_i\subset J$ and a subset 
$\Delta_i\subseteq \Omega(J_i)$ such that the $J_i$ are
disjoint, $i < j$ for all $j\in J_i$, and $\pi_i(\calG_{\Delta_i})\in\calR$.
It is clear then, that $\Delta :=\bigcup_{i\in I} \Delta_i$ accomplishes
our goal.

Suppose the $J_{\ell}$ have already been constructed for $\ell < i$.
Let $K_i$ be the complement in $J$ of the set
$\{t\in J\mid t\le i\}\cup\bigcup \{J_{\ell}\mid \ell < i\}$.
We perform the following algorithm to construct $J_i$ and $\Delta_i$.
The algorithm will gradually reduce $G_i$ until it becomes an
element of $\calR$, using the 2-color-reducibility of $\calQ$ to
$\calR$.  The variable $F$ stores the current
group $G_i$.  The $\nextt(K_i,m)$ operation produces the
smallest element of $K_i$ that is greater than $m$.

\mn
\indent 01 \quad $F :=G_i$\\
\indent 02 \quad $J_i := \emptyset$ \\
\indent 03 \quad $m := 0$ \\
\indent 04 \quad \bwhile\ $F\notin \calR$ \\
\indent 05 \quad \qquad $m:=\nextt(K_i,m)$\\
\indent 06 \quad \qquad $J_i := J_i \cup \{m\}$ \\
\indent 07 \quad \qquad let $\Psi_m \subseteq\Omega_m$ such that \\
\indent 08 \quad \qquad\qquad  $\vf_{im}((\vf_{im}^{-1}(F))_{\Psi_m}) < F$ \\ 
\indent 09 \quad \bend(\bwhile) \\
\indent 10 \quad $\Delta_i := \bigcup_{j\in J_i} \Psi_j$ \\
\indent 11 \quad \breturn\ $J_i$ and $\Delta_i$

\mn
Explanation of lines 07--08.  Both $G_i$ and $G_m$ belong to $\calQ$,
and therefore their subgroups $F$ and $\vf_{im}^{-1}(F)$, resp., also
belong to $\calQ$.  The restriction of the epimorphism
$\vf_{im} : G_m\to G_i$ to $\vf_{im}^{-1}(F)$ is an epimorphism
from $\vf_{im}^{-1}(F)$ onto $F$.  Therefore, if $F\notin \calR$,
by the 2-color-reducibility of $\calQ$ to $\calR$ there exists
$\Psi_m\subseteq \Omega_m$ as required in line 08.

\mn
Since $F$ is reduced in every round of the \bwhile-loop, the
process terminates in a finite number of steps, and on termination,
$F\in\calR$, as desired.
\end{proof}

\mn
We shall use Lemma~\ref{lem:reduction} in each successive
step in a chain of HS-closed classes which we now list.  
\begin{itemize}
\item \quad ${\mathscr{Gr}} = \{\text{ all finite groups }\}$
\item \quad ${\mathscr{Sol}} = \{\text{ all finite solvable groups }\}$
\item \quad ${\mathscr{Der}_k} = \{\text{ all finite solvable groups
    of derived length }\le k\,\}$
\end{itemize}
For some constant $k_0$ we shall descend along the chain
\begin{equation}   \label{eq:chain-of-classes}
 {\mathscr{Gr}} \supset {\mathscr{Sol}} \supset
{\mathscr{Der}_{k_0}} \supset {\mathscr{Der}_{k_0-1}}
\supset \dots \supset {\mathscr{Der}_1} \supset {\mathscr{Der}_0}\,.
\end{equation}

\mn
Note that $\mathscr{Der}_0$ consists only of the
trivial group, so once that class has been reached,
we have found an asymmetric coloring.

So we have reduced the proof of Tucker's Conjecture to the
following result.

\begin{theorem}  \label{thm:program}
  There exists a positive integer $k_0$ such that the
  following holds.\\
  Let $\calQ\supset\calR$ be a pair of consecutive terms in
  the chain~\eqref{eq:chain-of-classes}. Then
  $\calQ$ is 2-color-reducible to $\calR$.
\end{theorem}

The rest of the paper describes the proof of this result.

\section{Definitions, notation: group theory}    \label{sec:defs}
For the rest of this paper, {\bf all groups will be finite},
except where expressly stated otherwise.

We use the notation $[n]=\{1,\dots, n\}$ for integers $n\ge 0$.\\

For groups $G,H$, the notation $H\le G$ indicates that $H$ is
a subgroup, and $H < G$ indicates a proper subgroup.
The notation $N\normal G$ indicates a (not necessarily proper)
normal subgroup.
An \emph{epimorphism} (surjective homomorphism)
from $G$ onto $H$ is indicated as $G\onto H$.
For $K\subseteq G$ a subset of the group $G$ we
denote the centralizer of $K$ in $G$ by $\ccc_G(K)$.
The center of $G$ is $Z(G)= \ccc_G(G)$.

Let $G$ be a group.
The \emph{commutator} of $h,k\in G$ is the element
$[h,k]=h^{-1}k^{-1}hk$.  If $H,K\le G$ then
$[H,K]$ denotes the subgroup generated by all commutators
$[h,k]$ for $h\in H$, $k\in K$. 
The \emph{commutator subgroup} or \emph{derived subgroup} of
the group $G$ is $[G,G]$, also denoted $G'$.  The members of
the \emph{derived series} are denoted $G^{(k)}$ where
$G^{(0)}=G$ and $G^{(k+1)}=(G^{(k)})'$.  A group $G$ is
\emph{perfect} if $G'=G$.  Every finite group $G$ contains
a unique largest perfect subgroup, called the
\emph{perfect core} of $G$, reached when the
derived series stabilizes: $G^{(k+1)}=G^{(k)}$.
We denote the perfect core of $G$ by $G\inff$.\
The group $G$ is solvable if and only if its perfect core is
the identity.

The \emph{socle} of a group $G$, denoted $\soc(G)$,
is the product of its minimal normal subgroups.

An \emph{almost simple} group is a group $G$ of the form
$N\normal G\le \aut(N)$ where $N$ is a nonabelian simple group.
In this case, $N$ is the unique minimal normal subgroup of $G$
and therefore it is the socle of $G$.

The group $\inn(G)$ of inner automorphisms of $G$
consists of the conjugations by elements of $G$.
The \emph{outer automorphism group} of a group $G$
is the quotient $\out(G)=\aut(G)/\inn(G)$.
\emph{Schreier's Hypothesis} states that the
outer automorphism groups of all finite simple groups
are solvable.  This is equivalent to saying that the
perfect core of an almost simple group is simple.
Schreier's Hypothesis is a known consequence of the 
Classification of Finite Simple Groups.

For a group action $G\acts\Omega$ (see Def.~\ref{def:action})
we usually reserve the letter $n$ for $|\Omega|$, the \emph{degree}
of the action.  For additional definitions and notation about group
actions, see Section~\ref{sec:def1}.

\section{Solvable colorings of primitive groups}
\label{sec:solvable}

In this section we build one of our main tools for the
proof of Tucker's conjecture.

\begin{definition}
For a permutation group $G\le \sym(\Omega)$, let $\solv(G)$
denote the minimum number $k$ of colors such that $G$
admits a solvable $k$-coloring.  We call $\solv(G)$ the
\emph{solvable coloring number} of $G$.
\end{definition}

\begin{obs}   \label{obs:ps-closed}
Let $G,H \le\sym(\Omega)$.    
\begin{itemize}
\item[(a)] If $H\le G$ then $\solv(H)\le \solv(G)$.
\item[(b)] Let the orbits of $G$ be $\Omega_1,\dots,\Omega_k$
and let $G_i$ be the restriction of $G$ to $\Omega_i$.
Then $\solv(G)=\max_i \solv(G_i)$.
\item[(c)] For $x\in\Omega$, let $G(x)$ denote the action of
  the stabilizer $G_x$ on $\Omega\setminus {x}$.  Then
  $\solv(G(x))\le \solv(G)$.
\end{itemize}    
\end{obs}
\begin{proof} (a) A solvable coloring for $G$ is also a solvable
  coloring for any subgroup of $G$.\quad  (b) We can color
  each orbit independently, observing that $G$ is a subdirect
  product of the $G_i$. \quad (c) First, $\solv(G_x)\le\solv(G)$
  by (a).  Second, $\solv(G(x))=\solv(G_x)$ by (b).
\end{proof}
\begin{remark} \label{rem:ps-closed}
This observation would remain true if we replaced ``solvable groups''
by any PS-closed class of groups (closed under direct products
and subgroups) such as nilpotent groups, $2$-groups, groups
with composition factors of order $\le c$ or
solvable groups of derived length $\le c$ for some constant $c$.
\end{remark}

\bigskip

\begin{theorem}[Solvable coloring number of primitive groups]
   \label{thm:prim3col}
Let $G\le \sym(\Omega)$ be a primitive permutation group
of degree $n=|\Omega|$.  
\begin{itemize}
\item[(I)] \ If $G$ is solvable then $\solv(G)=1$.
\item[(II)] \ If $\alt(\Omega)\le G\le \sym(\Omega)$ then
     $\solv(G) = \lceil n/4 \rceil$.
\item[(III)]\   In all other\footnote{In
  a previous version of this paper, I listed $M_{24}$ as a
  possible exception.  I am grateful to
  Saveliy Skresanov~\cite{skresanov}
  for pointing out that the case of $M_{24}$ was settled by
  Chang Choi in 1972~\cite{choi}.  In the previous version I
  proved $\solv(M_{24})\le 3$, which suffices for our main results.}
  cases, $\solv(G)=2$.
\end{itemize}
\end{theorem}

Items (I) and (II) are straightforward.  (Regarding (II),
each color must occur at most 4 times.)
We discuss the Mathieu groups in Sec.~\ref{sec:mathieu}.  
The bulk of Section~\ref{sec:solvable} concerns the proof
of item (III).  

It was shown by Cameron, Neumann, and Saxl~\cite{saxl} in 1984
by a simple counting argument that all but a finite number of
primitive groups admit an asymmetric 2-coloring.  Gluck~\cite{gluck}
(1983) and Seress~\cite{seress97} (1997) classified the exceptions:
Gluck for the case when $G$ is solvable and Seress for 
the non-solvable case.  Seress gives their combined
list of 43 groups~\cite[Theorem 2]{seress97}.
We shall rely on Seress's list of non-solvable
exceptions.

\begin{remark}
For the proof of our main result, solvable colorings 
of almost simple primitive groups are not required.
We include this part of the result for completeness.
In particular, the hard part of Seress's result,
the classification of those almost simple primitive
groups that admit an asymmetric 2-coloring, is not
required.
\end{remark}

We begin with a straightforward but useful observation.

\begin{obs}  \label{obs:ptwise}
Let $H\le \sym(\Omega)$ be a permutation group
and $\Phi\subseteq\Omega$ an $H$-invariant subset
such that 
(a) the image of the action $H\acts\Phi$
is solvable and (b) the pointwise stabilizer 
$H_{(\Phi)}$ is solvable.
Then $H$ is solvable.  In particular, condition (a)
is met if $|\Phi|\le 4$.  \hfill $\Box$ 
\end{obs}

First we consider three particular classes of primitive groups:
affine groups, the projective linear groups, and the
Mathieu groups.  Our proof for these classes is self-contained
and does not rely on Seress's work.  (See the footnote for $M_{24}$.)

\subsection{Affine groups}
\label{sec:affine2col}
\begin{prop}
Let $G$ be a primitive permutation group with an
elementary abelian normal subgroup.  Then $G$ admits
a solvable 2-coloring.
\end{prop}
\begin{proof}
We need to find a subset $\Delta\subseteq\Omega$
such that $G_{\Delta}$ is solvable.  

\mn
Let $N$ be the elementary abelian normal subgroup; so
$n=|N|=|\Omega|=p^d$ for some prime $p$ and $d\ge 1$.
$\Omega$ can be viewed as the $d$-dimensional vector 
space over $\fff_p$, and $G\le\agl(d,p)$ (the affine
group acting on $\Omega$).  

\mn
If $d=1$ then $G$ is solvable, so $\Delta=\emptyset$ will do.

\mn
Let now $d\ge 2$.  Let $e_0=0$ and let $e_1,\dots,e_d$
be a basis of $\Omega$.

\mn
If $d=2$ then let $\Delta=\{e_0,e_1\}$.
The pointwise stabilizer 
$G_{(\Delta)}$ consists of linear transformations
of $\Omega$, described by triangular matrices (in the basis $\{e_1,e_2\}$),
hence $G_{(\Delta)}$ is solvable.  An application of
Obs.~\ref{obs:ptwise} shows that $G_{\Delta}$ is solvable.

\mn
For $d\ge 3$ we observe that
$G$ preserves affine relations, \ie, relations of the form
$\sum \alpha_i x_i=0$ where $x_i\in\Omega$, $\alpha_i\in\fff_p$,
and $\sum \alpha_i=0$.

\mn
By a quadruple we shall mean a set of four elements.
We shall say that quadruple $Q\subseteq\Omega$
satisfies the equation $f(x_1,x_2,x_3,x_4)=0$
if there is a bijection $\beta: \{1,2,3,4\}\to Q$
such that $f(\beta(1),\beta(2),\beta(3),\beta(4))=0$.

\mn
If $3\le d \le 6$ then let 
$\Delta = \{e_0,e_1,e_2,e_1+e_2, e_3, \dots, e_d\}$.
There is exactly one quadruple of
elements of $\Delta$ satisfying the equation
$x_1+x_2-x_3-x_4=0$, namely, $\{e_1,e_2,e_0,e_1+e_2\}$.
This means $G_{\Delta}$ fixes this quadruple (setwise)
and also the set $\{e_3,\dots,e_d\}$ (setwise).
On the other hand, $G_{(\Delta)}=1$.  So an application
of Obs.~\ref{obs:ptwise} to $G_{\Delta}$
shows that $G_{\Delta}$ is solvable.

\mn
Assume $d\ge 7$ and let $k=\lfloor (d-1)/2\rfloor$.
Let $\Delta = \{(-1)^{i+1} e_i\mid 0\le i\le d\}\cup
\{e_{2i-1}+e_{2i}+e_{2i+1}\mid 1\le i\le k\}$.
Consider the set $\calH$ of those quadruples
in $\Delta$ that satisfy the equation
$x_1-x_2+x_3-x_4=0$.  These are exactly the quadruples
$Q_i:=\{e_{2i-1}, -e_{2i}, e_{2i+1}, e_{2i-1}+e_{2i}+e_{2i+1}
\mid 1\le i\le k\}$.

Let us now consider the graph with vertex set $\Delta$ where
two vertices are adjacent if there is a quadruple in $\calH$
in which both of them participate.  This graph is invariant
under $G_{\Delta}$.  The graph has one or two isolated
vertices ($e_0$ and, if $d$ is even, $e_d$) 
and otherwise consists of a chain of
4-cliques, each one sharing one vertex with the next one.
The automorphism group of this graph is easy to determine;
it has a normal $2$-subgroup of index $9$,
therefore it is solvable.  On the other hand, $G_{(\Delta)}=1$,
so $G_{\Delta}$ acts faithfully on our graph and therefore
it is solvable.
\end{proof}

\subsection{Projective linear groups}
\label{sec:proj2col}
\begin{prop}
For $d\ge 2$ and a prime power $q$, the projective linear group
$L_d(q)$ in its natural action on the $(d-1)$-dimensional projective
space over $\fff_q$ admits a solvable 2-coloring.  
\end{prop}
\begin{proof}
For $d=2$, the stabilizer of any point is the $1$-dimensional
affine group, which is solvable.

For $d\ge 3$, let $e_i$ denote the standard unit vectors in
$\fff_q^d$ (the $i$-th coordinate $1$, the other coordinates $0$).
For $v\in \fff_q^d\setminus\{0\}$ we write 
$[v]=\{\lambda v\mid \lambda\in\fff_q^{\times}\}$ for the equivalence
class representing a point in $PG(d-1,q)$ by its homogeneous coordinates.

For $d=3$, let $\Delta=\{[e_1], [e_2], [e_3], [e_1+e_2+e_3]\}$.
The pointwise stabilizer of $\Delta$ in $N := L_3(q)$ is the identity;
therefore, the setwise stabilizer is $N_{\Delta}\le \sym(\Delta)$
which is solvable.

For $d\ge 4$, let
$\Delta=\{[e_i]\mid 1\le i\le d\}\cup
        \{[e_i+e_{i+1}]\mid 1\le i\le d-1\}\cup
        \{[f]\}$
where $f=\sum_{i=1}^d e_i\}$.
The pointwise stabilizer of $\Delta$ is the idenity, so we
only need to consider what permutations of $\Delta$ are
feasible under $L_d(q)$.  The action of $L_d(q)$ preserves
the underlying matroid (\ie, it maps linearly independent
sets to linearly independent sets).

The only linearly dependent triples in $\Delta$ are the triples
of the form $\{[e_i], [e_{i+1}], [e_i+e_{i+1}]\}$, and in
the case of $d=4$, the triple $\{[e_1+e_2], [e_3+e_4], [f]\}$.

So the 3-uniform hypergraph $\calH$ formed by these triples is preserved
by $L_d(q)$.  Let us find the degree of each element of $\Delta$
in this hypergraph (\ie, how many times each element of $\Delta$
appears in these triples).  Let
$\Delta_j = \{u\in\Delta\mid \deg_{\calH}(u)=j\}$.
Each $\Delta_i$ is fixed (setwise) by $N_{\Delta}$.

For $d=4$ we have $\Delta_1=\{[e_1], [e_4], [e_2+e_3], [f]\}$
and $\Delta_2 =\{[e_2], [e_3], [e_1+e_2], [e_3+e_4]\}$.
Therefore $N_{\Delta}\le \sym(\Delta_1)\times \sym(\Delta_2)$, solvable.

For $d\ge 5$ we have $\Delta_0=\{[f]\}$,\
$\Delta_1 = \{[e_1], [e_d]\}\cup \{[e_i+e_{i+1}]\mid 1\le i\le d-1\}$,
and $\Delta_2 = \{[e_2],\dots,[e_{d-1}]\}$.
Let us define the graph $R$ on vertex set $\Delta$ by making 
$u, v\in\Delta$ adjacent if $u\neq v$ and $\{u,v\}$ is a
subset of a triple in $\calH$.  The induced subgraph $R[\Delta_2]$
is the path $[e_2]--\dots--[e_{d-1}]$, which has only 2 automorphisms,
so the pointwise stabilizer of $\Delta_2$ has index $\le 2$ in
$N_{\Delta}$.  Moreover, this pointwise stabilizer also fixes
each point that has two neighbors in this path, so it
can only swap the pair $([e_1],[e_1+e_2])$ and the pair
$([e_d],[e_{d-1}+e_d])$.  In sumary, the order of
$N_{\Delta}$ divides $8$, so $N_{\Delta}$ is solvable.
\end{proof}          

\subsection{Mathieu groups}
\label{sec:mathieu}

\begin{prop}[Choi]  \label{prop:mathieu}
Each of the five Mathieu groups admits a solvable 2-coloring.
\end{prop}
\begin{proof}
We give a simple direct proof in the cases other than
$M_{24}$ and refer Choi~\cite{choi} for $M_{24}$.

Let $G$ be one of the Mathieu groups
$M_{23}$, $M_{22}$, $M_{12}$ and $M_{11}$.
So we can write 
$G=M_{m+k}$ where $m\in\{10, 21\}$ and $k=1,2$.
Let $G$ act on $\Omega$ where $|\Omega|=m+k$.
Let $\Delta\subset\Omega$ be a set of $2+k$ elements,
so $|\Delta|\le 4$.
Then the order of the pointwise stabilizer of $\Delta$ is 
$|G_{(\Delta)}|=48$ if $m=21$ and $8$ if $m=10$.
Therefore $G_{(\Delta)}$ is solvable, so by
Obs.~\ref{obs:ptwise}, $G_{\Delta}$ is solvable.

The remaining case, $G=M_{24}$, was settled by
Chang Choi in 1972~\cite{choi}.
Choi classified all setwise stabilizers of $M_{24}$.
He found a set of size 8
he denotes by $8'''$ such that $G_{8'''}$ has an
elementary abelian normal subgroup of order 16
with quotient $S_4$ \cite[Prop. 4.1]{choi}.
He also found a set of size 10 he denotes by $10'''$
such that $G_{10'''}\cong S_3\times S_4$ \cite[Prop. 6.3]{choi}.
\end{proof}

\begin{remark}  \label{rem:mathieu}
A set equivalent to $10'''$ was found by 
Saveliy Skresanov\footnote{Saveliy V. Skresanov,
Sobolev Institute of Mathematics, Novosibirsk.}~\cite{skresanov}
using the GAP computer algebra system, thus providing
independent verification of the fact that $\solv(M_{24}=2$.
Saveliy kindly agreed that I share his code.

\small{
\begin{verbatim}
gap> G := MathieuGroup(24);
Group([ (1,2,3,4,5,6,7,8,9,10,11,12,13,14,15,16,17,18,19,20,21,22,23),
(3,17,10,7,9)(4,13,14,19,5)(8,18,11,12,23)(15,20,22,21,16), (1,24)(2,23)
   (3,12)(4,16)(5,18)(6,10)(7,20)(8,14)(9,21)(11,17)(13,22)(15,19) ])
gap> S := [1..10];
[ 1 .. 10 ]
gap> StructureDescription(Stabilizer(G, S, OnSets));
"S4 x S3"
\end{verbatim}
}
\end{remark}

\begin{remark}
Of course $\solv(M_{24})=2$ implies $\solv(M_{23})=\solv(M_{22})=2$
(see Obs.~\ref{obs:ps-closed}),  
so those observations should also be attributed to Choi.

Let us note that $\solv(M_{23})=2$ immediately implies
$\solv(M_{24})\le 3$, which suffices for our main results.
Indeed, even the weaker statement $\solv(M_{24})\le 5$
would suffice: the only place in the proof of
Theorem~\ref{thm:simple1} (and consequently in the
proof of Theorem~\ref{thm:main})
where a bound on $\solv(M_{24})$ is used
is Case~2a of the proof of Theorem~\ref{thm:simple1}.
So the main results do not depend on Choi's
classification theorem.
\end{remark}

\subsection{Proof of Theorem~\ref{thm:prim3col}}

We refer to Seress's list of primitive groups that do not
admit an asymmetric 2-coloring~\cite[Theorem 2]{seress97}.
Seress lists 43 groups (this includes Gluck's list
of solvable exceptions); we organize the list into
four categories.
We write $n=|\Omega|$ for the degree of $G$.

\begin{theorem}[Seress]  \label{thm:seress}
Let $G\le\sym(\Omega)$ be primitive, $G\ngeq \alt(\Omega)$.
Assume $G$ does not admit an asymmetric 2-coloring.
Then $n\le 32$ and $G$ falls into one of the following categories.
\begin{itemize}
\item[(a)]  $G$ is solvable,
\item[(b)]  $G$ is an affine group, \ie, $G$ has an elementary
  abelian normal subgroup (so $n$ is a prime power),
\item[(c)]  $G$ has degree $n\le 8$\,,
\item[(d)]  $G$ is almost simple (so $N\le G\le \aut(N)$ for some
  nonabelian simple group $N$).
\end{itemize}
\end{theorem}

In Seress's list, Case (d) falls into the following subcategories.
The list indicates each group $G$ by the pair
$(n,N)$ where $n$ is the degree of $G$ (size of $|\Omega|$)
and $N$ is the (unique, simple) minimal normal subgroup of $G$.

\begin{itemize}
\item[(i)]  Mathieu groups in their natural action,
       $(n,M_n)$ for $n=11, 12, 22, 23, 24$
\item[(ii)]  projective linear groups $L_d(q)$ for some pairs
       $(d,q)$ where $d\ge 2$ and $q$ is a prime power,
       in their natural action on the $(d-1)$-dimensional
       projective space over $\fff_q$  $(n=(q^d-1)/(q-1))$
\item[(iii)] $(10,A_5)$, $(10, A_6)$, $(12, M_{11})$, $(15, A_8)$.
\end{itemize}
\begin{proof}[Proof of Theorem~\ref{thm:prim3col}]\quad \\
Given a primitive group $G\le \sym(\Omega)$ such that
$G\ngeq\alt(\Omega)$, we need to show that $G$ admits
a solvable 2-coloring, except if $G=M_{24}$ in its natural
action then we provide a solvable 3-coloring.

\mn
Case (o).\  If $G$ admits an asymmetric 2-coloring, that
is more than sufficient for us.  Now we need to eliminate
the finite number of exceptions.

\mn
In Case (a), the constant coloring $(\Delta=\emptyset)$
works.

\mn
Case (b).\  This case was settled in Sec.~\ref{sec:affine2col}.

\mn
Case (c).\ Now $5\le |\Omega|\le 8$.  Let $\Delta$ be any
4-subset of $\Omega$.  Then, applying Obs.~\ref{obs:ptwise}
to $H:= G_{\Delta}$, it follows that $G_{\Delta}$ is solvable. 

\mn
Case (d).\ Now $G$ is almost simple, so it has a
unique minimal normal subgroup $N$ which is 
nonabelian simple.  In particular, the quotient 
$G/N \le\out(N)$ is solvable by Schreier's
hypothesis.  It follows that it suffices
to find $\Delta\subseteq\Omega$ such that $M:=N_{\Delta}$
is solvable.  Indeed, let $L:=G_{\Delta}$.  Now $L$ is
solvable because $L/(L\capp N)$ is a subgroup of $G/N$.

\mn
Subcase (i) was settled in Section~\ref{sec:mathieu}.

\mn
Subcase (ii).\ 
In Section~\ref{sec:proj2col} we have shown that all projective
linear groups, in their natural action, admit solvable
2-colorings.

\mn
Subcase (iii).

\mn
$A_5$ is a minimal simple group, so all proper subgroups
are solvable; we can take any nontrivial subset of $\Omega$
for $\Delta$.

\mn
For $(10,A_6)$, the stabilizer of a point has order $36$
and is therefore solvable.

\mn
$A_8\cong L_4(2)$, and Seress's example $(15,A_8)$ is the
standard action of $L_4(2)$ on $PG(3,2)$, covered under Subcase (ii).

\mn
In the case of $G\cong M_{11}$ acting as a primitive group on
a set of size $|\Omega|=12$, take any four elements, $x,u,v,w\in \Omega$
such that the successive pointwise stabilizers strictly decrease:
$G > G_x > G_{(x,u)} > G_{(x,u,v)} > G_{(x,u,v,w)}$.  Let
$\Delta =\{x,u,v,w\}$.  Now $|G_x|=|G|/12=660=2^2\cdot 3\cdot 5\cdot 11$.
At each subsequent step, the order drops by at least a prime
factor, in total by at least a factor of $2\cdot 2\cdot 3=12$,
so the order of the pointwise stabilizer of $\Delta$ is 
$|G_{(\Delta)}|\le 660/12=55$.
Therefore $G_{(\Delta)}$ is solvable.  By Obs.~\ref{obs:ptwise},
it follows that $G_{\Delta}$ is solvable.
\end{proof}

\section{Proof of Theorem~\ref{thm:simple} for primitive groups}
\label{sec:primitive}

We restate this case.
\begin{lemma}
Let $G\le\sym(\Omega)$ be a primitive group, 
$T$ a nonabelian simple group, and $\vf : G\onto T$
an epimorphism.  Then there exists a subset
$\Delta\subseteq \Omega$ such that 
$\vf(G_{\Delta}) < T$.
\end{lemma}

\begin{proof}

\noindent
Case 1. \ $G$ is almost simple.  
In this case we claim that for any nontrivial subset 
$\Delta\subseteq\Omega$ ($\Delta\neq\emptyset$ and
$\Delta\neq\Omega$) we have $\vf(G_{\Delta}) < T$.

\mn
Indeed, in this case $G$ has a unique minimal normal 
subgroup $N$ which is nonabelian simple.  In particular,
the quotient $G/N=\out(N)$ is solvable by Schreier's
hypothesis.  Let $\ker(\vf)=K$.
Since $G/K\cong T$, it follows that $K\ngeq N$ and
therefore $K=1$, hence $G=N\cong T$.  So for any
nontrivial $\Delta$ we have $|G_\Delta| < |G|=|T|$ and
therefore $\vf(G_{\Delta}) < T$.

\mn
Case 2.\ $G$ is not almost simple.  In particular,
$G\ngeq \alt(\Omega)$.
In this case the result is an immediate consequence of
Theorem~\ref{thm:prim3col}.  Indeed, now we are in
Case (III) 
of Theorem~\ref{thm:prim3col}, so
$G$ admits a solvable 2-coloring, \ie, there is a subset
$\Delta\subseteq\Omega$ such that $G_{\Delta}$ is solvable,
and therefore $G_{\Delta}$ has no epimorphism on $T$.
\end{proof}

\begin{remark}
Note that while this proof rests on Seress's classification
of the primitive groups that do not admit an asymmetric 2-coloring,
it avoids any reference to the most difficult part of
Seress's work, the classification of the almost simple groups.
\end{remark}

\section{Reducing the image: Proof of Theorem~\ref{thm:simple}}
\label{sec:general2simple}

\subsection{Three lemmas}

The following lemma may be folklore.  
It appears as \cite[Lemma 8.1.1]{quasipoly}
along with an elegant proof, supplied by P\'eter P. P\'alfy
and reproduced below for completeness.
As remarked there, the result can also be derived
from \cite[Lemma 2.8]{meierfrankenfeld}.

\begin{lemma}[Subdirect product lemma] \label{lem:subdirect}
Let $G$ be a subdirect product of the finite groups $H_i$
($i=1,\dots,r$).
Let $\pi_i : G\onto H_i$ be the corresponding projections.
Let $\vf: G\onto T$ be an epimorphism, where
$T$ is a nonabelian simple group.  Let $K=\ker(\vf)$
and $M_i=\ker(\pi_i)$.
Then $(\exists i)(M_i\le K)$.
\end{lemma}
\begin{proof}[Proof by P\'eter P. P\'alfy]
For subgroups $G_1,\dots,G_k\le G$ we use the
notation
\[ [G_1,\dots,G_k] =[\dots [[G_1,G_2],G_3],\dots, G_k]\,.\]
Assume for a contradiction that
$K \ngeq M_i$ for all $i$.   Then $M_iK = G$ (because $K$ is a maximal
normal subgroup).   It follows that
$[G,\dots,G]=[M_1K,\dots,M_mK]\le K[M_1,\dots,M_m]\le
   K\left(\bigcap_{i=1}^m M_i\right) = K$,
so $[G/K,\dots,G/K]=1$, a contradiction because $G/K\cong T$ is
nonabelian simple.
\end{proof}

\begin{obs}[Perfect core lemma]    \label{obs:core}
Let $\vf: G\onto H$ be an epimorphism of finite groups.
Then $\vf(G')=H'$.
Consequently, $\vf(G\inff)=H\inff$.
\hfill $\Box$ 
\end{obs}

\begin{lemma}[Three normal subgroups lemma]   \label{lem:3normal}
Let $H$ be a finite group and $A,B,C$ three normal subgroups
satisfying the following conditions.
\begin{itemize}
\item[(i)] $AB=AC=BC=H$. 
\item[(ii)] $H/A$ and $H/B$ are nonabelian simple.
\end{itemize}
Then $B\ngeq A\capp C$.
\end{lemma}
\begin{proof}
Let $S=H/A$ and $T=H/B$.  So $S$ and $T$
are nonabelian simple groups.

Without loss of generality we may assume $A\capp B\capp C=1$.

Assume for a contradiction that $B\ge A\capp C$.  So
$1 = A\capp B\capp C=A\capp C$.

Since $A\capp C=1$ and $AC=H$, we have $H=A\times C$.
Therefore $C\cong H/A\cong S$ and $A\le \ccc_H(C)$.

We claim that $B\capp C=1$.  Indeed, given that
$C\cong S$ is simple, the alternative would be
$C\le B$.  But this would mean $H=BC=B$,
impossible by Assumption (ii).

Since $B\capp C=1$, we have $B\le \ccc_H(C)$.

Combining this with $A\le \ccc_H(C)$ we obtain that
$H=AB\le \ccc_H(C)$, \ie, $C\le Z(H)$.  
But then $C$ must be abelian,
contradicting the isomorphism $C\cong S$.
\end{proof}

\subsection{The proof of Theorem~\ref{thm:simple}}\quad

\mn
We restate and augment Theorem~\ref{thm:simple}.

\begin{theorem}   \label{thm:simple1}
Let $G\le \sym(\Omega)$, where $\Omega$ is a finite set.
Let $\vf: G\onto T$ be an
epimorphism where $T$ is a nonabelian simple group.  Then 
$(\exists \Delta \subseteq\Omega)(\vf(G_\Delta) < T)$.
Moreover, $\Delta$ can be chosen
to be a subset of one of the orbits of $G$.
\end{theorem}
\begin{proof}   
Let $n=|\Omega|$.
We fix $T$ and proceed by induction on $n$.
The statement is vacuously true if $G$ is solvable;
in particular, if $n\le 4$. Assume now $n\ge 5$.

Let $K=\ker(\vf)$.

1. First assume $G$ is intransitive, with orbits 
$\Omega_1,\dots,\Omega_{r}$.  Let $H_i$ denote
the restriction of $G$ to $\Omega_i$, so $H_i\le\sym(\Omega_i)$
and $G$ is a subdirect product of the $H_i$.
Let $M_i$ denote the kernel of the epimorphism $\pi_i:G\onto H_i$,
\ie, the pointwise stabilizer of $\Omega_i$ in $G$.
By the Subdirect product lemma (Lemma~\ref{lem:subdirect}), there 
is an $i\le r$ such that $M_i\le K$. This means that $\vf$ induces 
an epimorphism $\vfbar: H_i\onto T$.  
Note that $\vf$ is
the composition of $\pi_i$ and $\vfbar$\,:

\begin{equation}    \label{eq:orbits}
\xymatrix{
  G \ar[rr]^{\vf} \ar[dr]^{\pi_i} & & T \\
    &   H_i \ar[ur]^{\vfbar}
}
\end{equation}

By induction on $n$,
we find a subset $\Delta\subseteq \Omega_i$ such that 
$\vfbar((H_i)_\Delta) < T$.  
By restricting $G$ to $\Omega_i$ we see that
$(H_i)_\Delta=(\pi_i(G))_\Delta = \pi_i(G_\Delta)$.

\mn
Now by diagram~\eqref{eq:orbits},
$\vf(G_\Delta)=\vfbar(\pi_i(G_\Delta))=\vfbar((H_i)_\Delta) < T$,
 as desired.

\mn
 
 2. Assume now that $G$ is transitive, imprimitive.  
Let $\calB=\{B_1,\dots,B_k\}$
be a minimal system of imprimitivity (the $B_i$ are maximal blocks
of imprimitivity).  So the induced action
$\psi: G \acts\calB$ is primitive.

Let $N$ be the kernel of this action; so $N$ fixes each block $B_i$ 
(setwise) and we can naturally identify $\psi(G)$ with $G/N$.

\mn
Case 1.\quad $K\ge N$.  In this case we have the commutative diagram

\begin{equation}    \label{eq:case1}
\xymatrix{
  G \ar[rr]^{\vf} \ar[dr]^{\psi} & & T \\
    &   G/N \ar[ur]^{\vfbar}
}
\end{equation}

So we have the epimorphism $\vfbar : G/N \onto T$
and $G/N$ is a primitive subgroup of $\sym(\calB)$.
Therefore, by the primitive case of Theorem~\ref{thm:simple},
proved in Section~\ref{sec:primitive}, there exists 
$\Deltabar\subseteq \calB$
such that $\vfbar((G/N)_{\Deltabar}) < T$.

Let $\Delta\subseteq\Omega$ be any subset that,
for all $i$, intersects $B_i$ if and
only if $B_i\in\Deltabar$.  Then
$\psi(G_{\Delta})\le (G/N)_{\Deltabar}$ and therefore
$\vf(G_{\Delta})\le \vfbar((G/N)_{\Deltabar}) < T$
and we are done.

\mn
Case 2.\quad $K\ngeq N$.  Since $K$ is a maximal normal subgroup of $G$,
this is equivalent to saying that $KN=G$ and therefore
$N/(N\capp K)\cong G/K \cong T$, so $\vf(N)=T$.
In particular it follows that $N$ is not solvable and
therefore

\begin{equation}  \label{eq:block5}
  |B_i|\ge 5\,.
\end{equation}
  
\mn
Moreover, $G/N = \psi(G)=\psi(KN)=\psi(K)\psi(N)=\psi(K)$.

\mn
Let $N_i$ denote the restriction of $N$ to $B_i$, so $N_i\le\sym(B_i)$.
So $N$ is a subdirect product of the $N_i$.  Let $M_i$ denote the
kernel of the epimorphism $\rho_i : N\onto N_i$, \ie, the pointwise
stabilizer of $B_i$ in $N$.  By the Subdirect product lemma
(Lemma~\ref{lem:subdirect}), there  is an $i\le k$ such that
\begin{equation}  \label{eq:Mi}
  M_i\le K\,.
\end{equation}  
This means that $\vf$ induces an epimorphism
$\sigma_i: N_i\onto T$. 
Let $\wt\vf: N\onto T$ denote the restriction of $\vf$ to $N$.
Note that $\wt\vf$ is the composition of $\rho_i$ and $\sigma_i$:

\begin{equation}    \label{eq:case2}
\xymatrix{
  N \ar[rr]^{\wt\vf} \ar[dr]^{\rho_i} & & T \\
    &   N_i \ar[ur]^{\sigma_i}
}
\end{equation}

We may assume $i=1$.

(Actually, given that $G$ acts transitively on the blocks and
therefore on the $M_i$, we in fact have that $M_i \le K$ for all $i$,
hence the diagram~\eqref{eq:case2} holds for all $i$.  But we shall
not need this fact.)

Since $N \le G_{B_1}$, we have $\vf(G_{B_1})=T$.

We now split the proof into two cases.

\mn
Case 2a. 
\quad The action $G/N\acts \calB$ admits
a solvable 5-coloring.

\mn
By Theorem~\ref{thm:prim3col}, this case covers all primitive groups
except the alternating groups $A_k$ in their natural action for $k \ge 21$.
(Actually, we can claim a solvable 3-coloring in all these cases
except $A_k$ for $k\ge 13$.  We shall not use the full force of
Theorem~\ref{thm:prim3col}.)

\mn
By definition, in this case there exists a 5-coloring
$\gamma_0 : \calB\to [5]$ such that the group $(G/N)_{\gamma_0}$
is solvable.  Let $\gamma_1$ be the lifting of $\gamma_0$
to $\Omega$, \ie, for $x\in B_i$ we set $\gamma_1(x)=\gamma_0(B_i)$.
So under $\gamma_1$, each block is monochromatic.
Moreover, $\psi(G_{\gamma_1})=(G/N)_{\gamma_0}$.
Therefore, $N\cdot G_{\gamma_1}/N$ is solvable.

\mn
Let now $\Delta_1\subseteq B_1$ be such that
$\sigma_1((N_1)_{\Delta_1}) < T$.
Such $\Delta_1$ exists by induction on $n$.  Let $t:=|\Delta_1|$.
Let $\beta : [5] \to \{0,1,\dots,5\}\setminus \{t\}$
be an injection.  
For $i=2,\dots,k$ let $\Delta_i\subseteq B_i$
be an arbitrary subset of size $\beta(\gamma_0(B_i))$.
For this we need $|B_i|\ge 5$, which holds by Eq.~\eqref{eq:block5}.

\mn
Let $\Delta = \bigcup_{i=1}^k \Delta_i$.  What we have done
was coding the colors of $\gamma_0$ by the sizes of the subsets
$\Delta_i$, taking care to have $\Delta_1$ the only one among the
$\Delta_i$ to have size $t$.

\mn
We observe that $G_{\Delta}\le G_{\gamma_1}$
and therefore $NG_{\Delta}/N$ is solvable.

\mn
Claim A. \quad $\vf(G_{\Delta}) < T$.

\begin{proof}
Suppose for a contradiction that  $\vf(G_{\Delta}) = T$.
Then by Obs.~\ref{obs:core}, the perfect core of
$G_{\Delta}$ still maps onto $T$: \quad
$\vf((G_{\Delta})\inff) = T$.  But
$(NG_{\Delta})\inff/N=1$ (because $NG_{\Delta}/N$ is solvable),
hence $(G_{\Delta})\inff\le N$, so 
$\vf((G_{\Delta})\inff)=\sigma_1\rho_1(G_{\Delta})\inff)
\le \sigma_1((N_1)_{\Delta_1}) < T$ by the choice of $\Delta_1$,
a contradiction, proving Claim~A and thereby completing the proof
of the Theorem in Case 2a.
\end{proof}

\mn
Case 2b.\quad $G/N=\alt(\calB)$ or $\sym(\calB)$ and $k\ge 21$.

\mn
Let $G_1 :=G_{B_1}$.  If $\vf(G_1) < T$ then select $\Delta := B_1$
and we are done.  Henceforth we assume $\vf(G_1)=T$.
The kernel of the epimorphism $\vf_1 : G_1\onto T$
(the restriction of $\vf$ to $G_1$) is $K_1 :=K\capp G_1$.

\mn
Let $G^*$ denote the restriction of $G_1$ to $B_1$, so $G^*\le\sym(B_1)$.
Let $L$ denote the kernel of the epimorphism
$\pi : G_1\onto G^*$, so $L$ is the pointwise stabilizer
of $B_1$ in $G_1$, \quad $L=(G_1)_{(B_1)}$.

\mn
Case 2b1.\quad $L \le K_1$.\quad

\mn
Now $\vf_1 : G_1\onto T$ factors across
$G^*$:\quad $\vf_1 = \tau\pi$, hence $\tau(G^*)=T$.

\mn
Let $\Delta \subseteq B_1$ be such that $\tau((G^*)_{\Delta}) < T$.
Such a subset exists by induction on $n$.
Note that $\Delta$ is a nontrivial subset of $B_1$ 
since $\tau(G^*)=T$.

\mn
Claim B. \quad $\vf(G_{\Delta}) < T$.

\begin{proof}
Since $\Delta\subseteq B_1$ and $\Delta\neq\emptyset$,
we have $G_{\Delta}\le G_1$, so $G_{\Delta}=(G_1)_{\Delta}$.
Observe also that $\pi(G_1)_{\Delta}=(G^*)_{\Delta}$.
Therefore $\vf(G_{\Delta})=\vf_1((G_1)_{\Delta})=
\tau((G^*)_{\Delta}) < T$.
This proves Claim~B and thereby completes the proof of
Case 2b1.
\end{proof}

\mn
Case 2b2.\quad $L\nleq K_1$. 

\mn
Since $G_1/K_1\cong T$ is simple, $K_1$ is a maximal
normal subgroup in $G_1$, therefore $K_1L=G_1$.

\mn
Case 2b2i. \quad $N\le K_1$.

\mn
In this case we are in a situation analogous to
Case 1: $\vf_1$ factors across $N$, \ie, there exists
$\vfbar_1 : G_1/N \to T$ such that
$\vf_1 = \vfbar_1\psi_1$ where $\psi_1 : G_1\onto G_1/N$.
It follows that $\vfbar: G_1/N\onto T$  is an epimorphism,
but $G_1/N$ is symmetric or alternating of degree $\ge 20$,
so $\vfbar$ is an isomorphism.  Let $\Delta=B_1\cup \{x\}$
for an arbitrary $x\in B_2$.  Now $G_{\Delta}$ fixes
both $B_1$ and $B_2$ setwise, therefore $NG_{\Delta}/N$
has order less than $|NG_1/N|=|T|$, consequently
$|\vf(G_{\Delta})|=|\vfbar_1(NG_{\Delta}/N)| < |T|$
and we are done with Case 2b2i.

\mn
Case 2b2ii. \quad $N\nleq K_1$ and $G/N = \alt(\calB)$.\\
We claim that this case cannot occur.

\mn
Again because $K_1$ is a maximal normal subgroup of $G_1$,
we have $K_1N = G_1$.

\mn
Recall that $M_1 = N_{(B_1)}$, the pointwise stabilizer of
$B_1$ in $N$.  Moreover, we are in Case~2, so
$M_1\le K$ by Eq.~\eqref{eq:Mi}, and therefore $M_1\le K_1$.

\mn
Recall that $L=(G_1)_{(B_1)}$.  Therefore $M_1 = N\capp L$.
So if $L\le N$ then $L=M_1\le K_1$, 
contradicting the assumption that put us in Case 2b2.
Therefore $L\nleq N$.
But now $N$ is a maximal normal subgroup in $G_1$
(since $G_1/N\cong A_{k-1}$), so $NL=G_1$.

\mn
Summarizing, we have $NK_1=NL=K_1L=G_1$ and $G_1/N\cong A_{k-1}$
and $G_1/K_1\cong T$; these quotients are nonabelian simple.
Setting $H:=G_1$, $A:=N$, $B:=K_1$, and $C:=L$, all assumptions
of Lemma~\ref{lem:3normal} are satisfied.  The conclusion is
that $K_1\ngeq N\capp L$.  But $N\capp L=M_1$ and we know that
$M_1\le K_1$, a contradiction, proving that Case 2b2ii cannot
occur.

\mn
Case 2b2iii. \quad $N\nleq K_1$ and $G/N = \sym(\calB)$.

\mn
In this case we have $\psi(G)= \sym(\calB)$.  Let
$\wt{G} = \psi^{-1}(\alt(\calB))$.  Replace $G$ by
$\wt{G}$; then we land in Case~2b2ii, which is impossible.
This completes the proof of this last case and with it
the proof of Theorems~\ref{thm:simple1} and~\ref{thm:simple}.
\end{proof}

\section{Reduction to the solvable case}
\label{sec:red-to-solvable}

First we extend Theorem~\ref{thm:simple} to all non-solvable
target groups.
\begin{theorem}[Reducing non-solvable image]\
 \label{thm:nonsolvable}
Let $G\le \sym(\Omega)$.  Let $\vf: G\onto H$ be an epimorphism
where $H$ is a non-solvable group.  Then 
$(\exists \Delta \subseteq\Omega)(\vf(G_\Delta) < H)$.
\end{theorem}

\begin{obs}\   \label{obs:inf-epi}
If $\vf : G\onto H$ is an epimorphism then
$\vf(G')=H'$.  It follows that $\vf(G\inff)=H\inff$.  
\hfill $\Box$ 
\end{obs}

\begin{proof}[Proof of Theorem~\ref{thm:nonsolvable}]
Since $H$ is not solvable, 
there is an epimorphism $\psi : H\inff\onto T$ where
$T$ is nonabelian simple.  Composing $\psi$ with $\vf$ we obtain 
an epimorphism $\xi : G\inff \onto T$.  By Theorem~\ref{thm:simple}
there exists 
$\Delta\subseteq\Omega$ such that $\xi((G\inff)_\Delta) < T$
and therefore $\vf((G\inff)_\Delta) < H\inff$.
We claim that this set $\Delta$ satisfies the conclusion
of the Lemma, \ie, $\vf(G_\Delta) < H$.

Let $L=G_\Delta$.   Observe that 
$L\inff\le G\inff\capp L=(G\inff)_\Delta$.
Therefore $\vf(L\inff)\le \vf((G\inff)_\Delta)
 < H\inff$.  It follows (by Observation~\ref{obs:inf-epi})
that $\vf(L) < H$.
\end{proof}

\section{Reduction to bounded derived length}
\label{sec:bded-derived-length}

\begin{theorem}    \label{thm:bded-orbits}
There exists a constant $c$ such that the following holds.
If $G\le\sym(\Omega)$ is solvable then there exists
$\Delta\subset\Omega$ such that all orbits
of $G_{\Delta}$ have length $\le c$.  In particular,
the derived length of $G_{\Delta}$ is less than $2c$.
\end{theorem}

The bound $2c$ comes from \cite{chains} where it is shown
that the length of any subgroup chain in $S_n$ is less than $2n$.
A far stronger bound on the length of the derived chains
of permutation groups is known;
Dixon~\cite{dixon} has shown that the derived length of
a solvable group of degree $n$ is at most $(5\log_3 n)/2$.
However, we do not need any of this, all we need is that
groups of bounded order have bounded derived length, which
is obvious.

We devote the rest of this section to proving Theorem~\ref{thm:bded-orbits}.

We say that a $k$-coloring $\gamma: \Omega\to [k]$
is \emph{uniform} if $|\gamma^{-1}(i)|=n/k$ for each $i\in [k]$.
Note that in a non-uniform $k$-coloring we permit some of the colors
not to be used.

\begin{lemma}    \label{lem:solvprim5color}
Let $G\le \sym(\Omega)$ be a primitive solvable group.
Then $G$ admits a non-uniform asymmetric 5-coloring.
\end{lemma}
\begin{proof}
Gluck proves \cite[Theorem 1]{gluck} that for $n\ge 10$, 
$G$ admits an asymmetric 2-coloring.
We can view such a coloring as a non-uniform 3-coloring
by adding an empty color.
(In fact, Gluck proves the existence of a non-uniform
2-coloring for $n\ge 10$.)  

To address the remaining cases directly, we observe that,
since $G$ is primitive and solvable, we are in the affine case:
$\Omega$ can be identified with the vector space $\fff_p^d$
for some prime $p$ and $d\ge 1$, and $G\le \agl(d,p)$.

Let us assign distinct colors to $0$ and the $d$ vectors in
a basis, and one more color to the rest of the space.
This coloring is clearly asymmetric, and it uses
at most $d+2$ colors.  If $n\le 9$ then $d\le 3$, so
we are using at most $5$ colors.

This coloring is not uniform for any value $n\le 9$.
(Note that we throw in one or more empty color classes
for $n\le 4$.)
\end{proof}
\begin{remark}
The bound $5$ in the Lemma is tight, as shown by $G=S_4$.
\end{remark}

Now we extend this result to all solvable permutation groups,
dropping the condition of non-uniformity.

\begin{lemma}   \label{lem:solv5color}
Let $G\le \sym(\Omega)$ be a solvable group.
Then $G$ admits an asymmetric 5-coloring.
\end{lemma}

\begin{remark}  \label{rem:5colorstight}
The bound $5$ in the Lemma is tight, as shown by the wreath
product $G=S_4\wr H$ where $H$ is any nontrivial permutation
group.  ($H$ is the top group in the wreath product.)
\end{remark}

\begin{definition}
Let $G\le\sym(\Omega)$ and let $\gamma_1,\gamma_2 :\Omega\to\Sigma$
be two colorings.  We say that a permutation $\pi\in\sym(\Omega)$
is a \emph{$G$-isomorphism} of $\gamma_1$ to $\gamma_2$ if
for all $x\in\Omega$, $\gamma_2(\pi(x))=\gamma_1(x)$.
\end{definition}

\begin{obs}    \label{obs:permute-colors}
If a group $G$ admits a non-uniform asymmetric $k$-coloring
then by permuting the colors, we obtain at least $k$
non-$G$-isomorphic non-uniform $k$-colorings.
\hfill $\Box$ 
\end{obs}

\begin{proof}[Proof of Lemma~\ref{lem:solv5color}]
Without loss of generality we may assume $G$ is transitive.
(Otherwise, select a coloring of each orbit and combine them.)
  
We proceed by induction on $n=|\Omega|$.
If $G$ is primitive, we are done by 
Lemma~\ref{lem:solvprim5color}.

Let now $G$ be imprimitive and
let $\calB=\{B_1,\dots,B_k\}$ be a maximal system
of imprimitivity (the blocks are minimal).  Let $|B_i|=t$.

Let $G_i$ denote the restriction of the setwise stabilizer
$G_{B_i}$ to $B_i$, so $G_i\le\sym(B_i)$.  Since $B_i$
is a minimal block, $G_i$ is primitive.  The $G_i$ are
equivalent permutation groups (equivalent under the
action of $G$).  Let $\delta_{1,1},\dots,\delta_{1,5}$ be five
non-$G_1$-isomorphic non-uniform $G_1$-asymmetric colorings
of $B_1$, and let $\delta_{i,1},\dots,\delta_{i,5}$ be
colorings of $B_i$ obtained by applying an element
$\sigma_i$ to our colorings of $B_1$ where $\sigma_i(B_1)=B_i$.
(Let $\sigma_1$ be the identity.)  Let $\widehat{\delta}_{i,j}$
be the 6-coloring of $\Omega$ obtained by adding a 6th color
everywhere outside $B_i$.  It should be clear that
if $\widehat{\delta}_{i,j}$ is $G$-isomorphic to
$\widehat{\delta}_{i',j'}$ then $j=j'$, simply by counting the
multiplicities of colors.

Let $\wt{G}$ denote the image of the $G\acts\calB$ action,
so $\wt{G}\le\sym(\calB)$.  By induction,
let $\wt{\gamma} : \calB\to [5]$
be a non-uniform asymmetric 5-coloring for $\wt{G}$.

We encode these colors by non-isomorphic colorings
of the blocks: We define the 5-coloring
$\gamma : \Omega\to [5]$ by setting, for each
$x\in B_i$,\ $\gamma(x)=\delta_{i,j}(x)$ where $j=\wt{\gamma}(B_i)$.

Assume $\sigma\in G$ preserves the coloring $\gamma$.  Since
the coloring $\wt{\gamma}$ can be reconstructed from $\gamma$,
we infer that $\sigma(B_i)=B_i$ for every $i$.  Now 
$\sigma_{|B_i}\in G_i$, and the coloring $\gamma|_{B_i}=\delta_{i,j}$
is $G_i$-asymmetric, so $\sigma=1$.
\end{proof}

\begin{lemma}  \label{lem:many2colorings}  
  There exist $\epsilon > 0$
and a threshold $n_1$
such that the following holds.  If $G\le \sym(\Omega)$ is a 
primitive solvable permutation group of degree $n \ge n_1$
then for every $j$ in the interval $n^{1-\epsilon}\le j \le n/2$
there exists a $G$-asymmetric subset $\Delta_i\subseteq\Omega$
of size $|\Delta_j|=j$.
\end{lemma}
\begin{proof}
The proof consists of drawing the required conclusion from Gluck's
counting argument~\cite{gluck}.  We review the argument.

Since $G$ is primitive and solvable, we are in the affine case:
$\Omega$ can be identified with the $d$-dimensional vector
space over $\fff_p$ for some prime $p$ and $d\ge 1$,
and with this identification, $G\le \agl(d,p)$.  

It follows that the fixed points of any element of $G$
form an affine subspace of $\Omega$ and therefore the minimal 
degree of $G$ is $\ge n/2$ and consequently every
non-identity element $\sigma\in G$ decomposes into at
most $3n/4$ cycles.  It follows that the number of subsets 
fixed by $\sigma$ is at most $2^{3n/4}$ and therefore
the number of subsets that are not $G$-asymmetric is at most
$|G|\cdot 2^{3n/4}$.
Invoking the P\'alfy--Wolf Theorem~\cite{palfy, wolf}
that states that the order of a
primitive solvable group of degree $n$ is less than $n^C$
for some constant $C$, the stated conclusion follows.
\end{proof}  
We shall only use a weak corollary of this result.
\begin{corollary}    \label{cor:5asymm}
There exists a constant $c_0$ such that every primitive
solvable permutation group $G$ of degree $\ge c_0$ admits
at least five $G$-asymmetric subsets of different sizes.
\hfill $\Box$ 
\end{corollary}

We are now ready to prove the main result of this section.

\begin{proof}[Proof of Theorem~\ref{thm:bded-orbits}]
This proof is not by induction, but we follow the scheme
of the inductive step in the proof of Lemma~\ref{lem:solv5color}.
Like there, we may assume $G$ is transitive.

\mn
We shall refer to the constant $c_0$ from Cor.~\ref{cor:5asymm}.  
We claim that there exists a set $\Delta\subseteq\Omega$
such that the orbits of $G_{\Delta}$ have length $\le c:=3c_0$.

\mn
A $G$-asymmetric subset does more than required (orbits of length 1),
so if $G$ is primitive and $n\ge c_0$ then we are done by 
Cor.~\ref{cor:5asymm}.  In fact, by Gluck~\cite{gluck}, we are 
done for all $n\ge 10$.

Let now $G$ be imprimitive and
let $\calB=\{B_1,\dots,B_k\}$ be a system
of imprimitivity.  Let $|B_i|=t$.
Let $G_i$ denote the restriction of the setwise stabilizer
$G_{B_i}$ to $B_i$, so $G_i\le\sym(B_i)$.

Let $\wt{G}$ denote the image of the $G\acts\calB$ action,
so $\wt{G}\le\sym(\calB)$.  Using Lemma~\ref{lem:solv5color},
let $\gamma : \calB\to [5]$ be a $\wt{G}$-asymmetric
5-coloring of $\calB$.

\mn
Case 1.\quad $4\le t \le 3c_0$.  In this case we claim
there exists $\Delta\subseteq\Omega$ such that the orbits
of $G_{\Delta}$ have length $\le t\le 3c_0$.

\mn
In this case, let $\Delta_i$ be an arbitrary subset of
$B_i$ of size $\gamma(B_i)-1$.  Let $\Delta = \bigcup_{i=1}^k \Delta_i$.

Assume $\sigma\in G$ preserves the set $\Delta$.  Since
the coloring $\gamma$ can be reconstructed from $\Delta$
simply by the sizes of the $\Delta_i$,
we infer that $\sigma(B_i)=B_i$ for every $i$.

Therefore, the orbits of $G_{\Delta}$ are subsets of the $B_i$
and therefore have length $\le t < 3c_0$.  This completes
the proof in Case 1.

\mn
Case 2.\quad $t \ge c_0$ and 
the $B_i$ are minimal blocks of imprimitivity.
In this case we claim
there exists an asymmetric $\Delta\subseteq\Omega$,
so the orbits of $G_{\Delta}$ have length 1.

\mn
Let $j_1 <\dots < j_5$ be five different sizes of $G_1$-asymmetric
subsets of $B_1$.

Let $\Delta_i\subseteq B_i$ be a $G_i$-asymmetric subset of
size $j_{\gamma(B_i)}$.  Let $\Delta = \bigcup_{i=1}^k \Delta_i$.

Assume $\sigma\in G$ preserves the set $\Delta$.  Since
the coloring $\gamma$ can be reconstructed from $\Delta$
simply by the sizes of the $\Delta_i$,
we infer that $\sigma(B_i)=B_i$ for every $i$.  Now 
$\sigma_{|B_i}\in G_i$, and the set $\Delta_i = \Delta\capp B_i$
is $G_i$-asymmetric, so $\sigma=1$.
This completes the proof in Case 2.

\mn
Case 3.\quad In all the remaining cases, we claim
there exists $\Delta\subseteq \Omega$ such that
the orbits of $G_{\Delta}$ have length $\le t\le 3$.

\mn
Now there is no block of size $4\le t < 3c_0$ (Case~1)
and also no minimal block of size $t \ge c_0$ (Case~2).
This means all minimal blocks have size $2\le t\le 3$ 
and the next smallest blocks have size $\ge 3c_0$.

\mn
Let $\calB=\{B_1,\dots,B_k\}$ be a system of minimal blocks,
so $2\le |B_i|\le 3$.  As before, let $\wt{G}$ be the
image of the action $G\acts\calB$.

\mn
Let further $\calD=\{D_1,\dots,D_{\ell}\}$ be a system of
imprimitivity that is a maximal coarsening of $\calB$,
\ie, $\calB$ is a refinement of $\calD$ and the $D_j$
are minimal among the blocks of imprimitivity that strictly
include some $B_i$.  Let $|D_j|=st$.  Let
$\wt{D}_j = \{B_i\in\calB\mid B_i\subseteq D_j\}$.
Now the $\wt{D}_j$ are minimal blocks for $\wt{G}$
and their size is $s \ge c_0$.  Therefore, by Case~2,
we have a $\wt{G}$-asymmetric set $\wt{\Delta}\subseteq \calB$.

\mn
Let us lift $\wt{\Delta}$ to $\Omega$ by setting
$\Delta = \bigcup \{B_i \mid B_i\in\wt{\Delta}\}$.
Now $G_{\Delta}$ fixes each $B_i$ setwise, so
its orbits have length $\le 3$.  This 
completes the proof of Theorem~\ref{thm:bded-orbits}.
\end{proof}

\section{From bounded derived length to asymmetry:
  reducing the derived length}
\label{sec:reducing-derived-length}

We restate Prop.~\ref{prop:derivedlength-intro}

\begin{prop}   \label{prop:derivedlength}
Let $G\le \sym(\Omega)$ be a solvable group with derived length $k\ge 1$.
Then there exists a subset $\Delta\subseteq \Omega$ such that
the derived length of $G_{\Delta}$ is at most $k-1$.
\end{prop}

We begin with the abelian case.
\begin{obs}
All abelian permutation groups admit an asymmetric 2-coloring.
\end{obs}
\begin{proof}
We need to construct a $G$-asymmetric subset of
the domain $\Omega$.

Let $R_1,\dots,R_k$ be the orbits of $G$ and let $\Delta$ be a
transversal of the orbits, \ie, $\Delta\subseteq\Omega$ and
$|\Delta\capp R_i|=1$ for all $i$.\\
We claim that $|G_{\Delta}|=1$.  Indeed, the restriction of
$G$ to $R_i$ is a transitive abelian group which therefore is
regular, so fixing one of its points fixes the entire orbit
pointwise.
\end{proof}  

\begin{proof}[Proof of Prop.~\ref{prop:derivedlength}]

\mn
Let $H:=G^{(k-1)}$ be the last nontrivial term in the derived series
of $G$.  So $H$ is abelian.  Let $\Delta\subseteq\Omega$ be such
that $|H_{\Delta}|=1$.  We claim that the derived length of
$G_{\Delta}$ is at most $k-1$.

Indeed, let $L=G_{\Delta}$.  So $L^{(k-1)}\le L\capp G^{(k-1)} =
L\capp H = H_{\Delta} = 1$.
\end{proof}
This completes the proof of Theorem~\ref{thm:main1} and with it
the proof of our main result, Theorem~\ref{thm:main}.

\section{An effective version?}   \label{sec:effective}

In this section we consider finite inverse sequences, of
length $k$, of finite permutation groups.  Such a system is defined by
a sequence of $k+1$ groups, $(G_i : i=0,\dots,k)$ and a
sequence of $k$ homomorphisms, $\vf_i : G_i\to G_{i-1}$, $i=1\dots, k$.
We denote such a system as $(G_i,\vf_i)_{i\le k}$.  Here
$G_i\le \sym(\Omega_i)$.  For $i\ge j$, the transition homomorphism
$\vf_{i,j} : G_i \to G_j$ is defined as the composition of
$\vf_{\ell,\ell-1}$ for $\ell=i,\dots,j+1$.  Now the definition
of inverse limits applies; in particular, strands are defined 
as in Def.~\ref{def:strand} and they form the inverse limit
$\calG=\varprojlim G_i$.  We say that the inverse system, and
the inverse limit, \emph{ends} at $G_0$.
If the $\Omega_i$ are disjoint then
we view $\calG$ as a permutation group acting on 
$\Omega:=\bigcup_{i=0}^k \Omega_i$.
We say that a coloring $\gamma:\Omega\to\Sigma$ is
\emph{zero-asymmetric} if $\calG_{\gamma}$ fixes $\Omega_0$
pointwise.  We say that the coloring is \emph{zero-neutral}
if $\Omega_0$ is monochromatic, \ie, $|\gamma(\Omega_0)|=1$.
If all the $\vf_i$ (and therefore all the $\vf_{i,j}$)
are epimorphisms, we speak of an \emph{epimorphic
inverse sequence.}

Let $\mathscr{Gr}$ denote the class of
finite permutation groups.
\begin{theorem}[asymmetric coloring of inverse limit---finite version]
  \label{thm:finite}
There exists a positive integer $c$ and a function
$f: \mathscr{Gr}\to\nnn$ such that the following holds.
Let $(G_i,\vf_i)_{i\le k}$ be an epimorphic inverse sequence
of length $k\ge c$ of finite permutation groups with disjoint domains.
Assume $k\ge c+f(G_c)$.  Then the inverse limit  of this
system admits a zero-neutral zero-asymmetric 2-coloring.
\end{theorem}

It is easy to see that our main technical result, 
Theorem~\ref{thm:main1}, is a consequence of 
Theorem~\ref{thm:finite}.  (Zero-neutrality is not
needed for this inference.)  The proof follows the lines
of the proof of Lemma~\ref{lem:reduction}.

It is also easy to see that our proof of Theorem~\ref{thm:main1}
in effect proved Theorem~\ref{thm:finite}.
\begin{proof}[Sketch of proof of Theorem~\ref{thm:finite}]
Take $c:=2c'+1$ where $c'$ is the constant
denoted by $c$ in Theorem~\ref{thm:bded-orbits}.
Set $f(n):= \lfloor\log_2(n)\rfloor$.

We color the domains $\Omega_i$ one at a time,
reducing the group $G_i$ and thereby $\calG$.
We call such a step a \emph{round},
and we conclude each round by
\emph{epimorphic reduction}
(Fact~\ref{fact:epiredux}).

First we color the domains $\Omega_i$ for $i > c$
to reduce $G_c$ to a solvable group.
While $G_c$ is not solvable, we reduce it to a
proper subgroup by coloring the next $\Omega_i$
(Theorem~\ref{thm:nonsolvable}).  This process clearly
terminates in $\le \log_2 |G_c|$ rounds and when it
terminates, $G_c$ is solvable.  But then,
all the $G_j$,\ $j\le c$, are solvable as well, by
epimorphic reduction.

Next we color $\Omega_c$, applying Theorem~\ref{thm:bded-orbits}
to $G_c$ and achieving
derived length $\le 2c'=c-1$ for $G_c$ and therefore for all $G_i$
for $i\le c$. 
Then, for $i=c-1$ down to $i=1$, we successively reduce
the derived length of $G_i$ to $\le i-1$,
using the procedure of Section~\ref{sec:reducing-derived-length}.
In the end, the derived length of $G_1$ is reduced to zero,
hence the same is true for $G_0$ without having colored
$\Omega_0$, yielding zero-asymmetry and zero-neutrality.
\end{proof}    

To make this result more effective, we need to replace
the bound $k\ge c+f(G_c)$ by a bound of that only
depends on $G=G_0$.  
\begin{conjecture}[asymmetric coloring of inverse limit---effective version]
 \label{conj:finite1}
 There exists a function
$g: {\mathscr{Gr}}\to\nnn$ such that the following holds.
Let $(G_i,\vf_i)_{i\le k}$ be an epimorphic inverse sequence of
length $k\ge g(G_0)$ of finite permutation groups with disjoint 
domains.  Then the inverse limit  of this
system admits a zero-neutral zero-asymmetric 2-coloring.
\end{conjecture}

If this conjecture is true, here is a lower bound on the
function $g$.  Let $\asy(G)$ denote the asymmetric coloring number
of the permutation group $G$ (Def.~\ref{def:asy}).
\begin{prop}
If Conjecture~\ref{conj:finite1} holds for a function $g$ then
for all finite permutation groups $G$ we have
\begin{equation}
  g(G) \ge \log_2 \asy(G).
\end{equation}
\end{prop}
\begin{proof}
Let $G\le \sym(\Omega)$.  
Consider the length-$k$ inverse system $(G,G,\dots,G)$ with the identity
serving as transition homomorphisms.  So the inverse limit $\calG$
is the diagonal of the direct product $G^{k+1}$, acting on
$\Omega\times \{0,\dots,k\}$.  Assume there is a zero-asymmetric
zero-neutral 2-coloring $\gamma$ of $\calG$.  Now define the
coloring $\delta$ of $\Omega$ by setting
$\delta(x)=(\gamma(x,1),\dots,\gamma(x,k))$ for $x\in\Omega$.
It should be clear that $\delta$ is an asymmetric coloring of $G$
with $\le 2^k$ colors, so $\asy(G)\le 2^{g(G)}$.
\end{proof}  

One might ask, how close is this lower bound to the true
upper bound (if one exists).  I would risk the following
bold conjecture.

\begin{conjecture}[asymmetric coloring of inverse limit---polylog bound]
  \label{conj:finite2}
There exists a polynomial $p$ such that
Conjecture~\ref{conj:finite1} holds with
$g(G)=p(\log(\asy(G)))$.
\end{conjecture}

Conjecture~\ref{conj:finite2} is true for inverse systems of
solvable groups.  Since solvable groups have bounded
asymmetric coloring number (Lemma~\ref{lem:solv5color}),
this means for solvable groups $G$ the quantity $g(G)$ should
be bounded, and indeed it is.

\begin{prop}
There exists a constant $c$ such that the following holds
for all inverse sequences of solvable permutation groups
with disjoint domains.  If the length of the sequence
is at least $c$ then the inverse limit admits a
zero-asymmetric zero-neutral 2-coloring.
We can take $c:=2c'+1$ where $c'$ is the constant
denoted by $c$ in Theorem~\ref{thm:bded-orbits}.
\end{prop}
\begin{proof}
We just proved this in the second (solvable) phase of the
proof of Theorem~\ref{thm:finite}.
\end{proof}

\section{Combinatorial relaxation of symmetry: CFSG-free proofs}
\label{sec:cfsg-free}
One of the key facts underlying our result was the following.
\begin{theorem}[\cite{saxl}]  \label{thm:saxl}
All but a finite number of primitive permutation groups,
other than the symmetric and alternating groups
in their natural action, admit an asymmetric 2-coloring.
\end{theorem}
The original proof of Theorem~\ref{thm:saxl}
rests on the Classification of Finite Simple Groups (CFSG).

In this section we address the following two questions.

\mn
\begin{itemize}
\item[($\alpha$)] 
Can one avoid the use of the CFSG in the proof of Theorem~\ref{thm:saxl}?
\item[($\beta$)]
Is there a combinatorial generalization of Theorem~\ref{thm:saxl},
\ie, an asymmetric 2-colorability result for
a class of combinatorial structures with no symmetry
assumptions, that includes Theorem~\ref{thm:saxl}?
\end{itemize}
Question ($\alpha$) was already raised by
Cameron \emph{et al.}~\cite{saxl}
and was reiterated by Imrich \emph{et al.}
as~\cite[Question 1]{watkins15}.

We point out that a positive answer to both questions
follows from a recent breakthrough
by Xiaorui Sun and John Wilmes~\cite{sunwilmes}
on the number of automorphisms of primitive coherent
configurations (see Theorems~\ref{thm:sunwilmes}
and~\ref{thm:sunwilmes2} below).

\subsection{CFSG-free proof of the Cameron--Neumann--Saxl Theorem}
\label{subsec:cfsg-free}
Recall that the \emph{line-graph} of a graph $X=(V,E)$
is the graph $L(X)$ with vertex set $E$ where two edges
$e,f\in E$ (as vertices of $L(X)$) are adjacent in $L(X)$
if they share a vertex in $X$.  The line-graphs of the cliques
$K_r$ are called \emph{triangular graphs}, denoted $T(r)$.
The graph $T(r)$ has $\binom{r}{2}$ vertices, $r!$ automorphisms,
and motion $2r-4$.  The socle has index 2 in $\aut(T(r))$.
The line-graphs of balanced bipartite cliques 
$K_{r,r}$ are called \emph{lattice graphs}, denoted $L_2(r)$.
The graph $L_2(r)$ has $r^2$ vertices, $2(r!)^2$ automorphisms,
and motion $2r$.  The socle has index 8 in $\aut(T(r))$.

The original proof of Theorem~\ref{thm:saxl} depends on CFSG through
the following result, a special case of much more detailed
result in~\cite{cameron81}.  Let $\soc(H)$ denote the socle
of the group $H$.  Let us say that $G$ is a \emph{top group}
if either $A_n\le G\le S_n$ or $\soc(H)\le G\le H$ where $H$ is
the automorphism group of a triangular graph or a lattice graph.
\begin{prop}[Cameron + CFSG]  \label{prop:cnsbound}
Let $G\le S_n$ be a primitive permutation group.
Assume $|G|\ge 2^{\sqrt{n/2}}$.  If $n$ is sufficiently
large then $G$ is a top group.
\end{prop}

The following elementary result by Sun and Wilmes~\cite{sunwilmes},
a consequence of their result on coherent configurations
(Theorem~\ref{thm:sunwilmes2}), implies Prop.~\ref{prop:cnsbound}.

\begin{theorem}[Sun--Wilmes, elementary]  \label{thm:sunwilmes}
There exists $c > 0$ such that the following holds.
Let $G\le S_n$ be a primitive but not doubly transitive
permutation group.
Assume $|G|\ge \exp(c(n^{1/3}(\log n)^{7/3}))$.
Then $G$ is a top group.
\end{theorem}

This result implies Prop.~\ref{prop:cnsbound}
except in the case that the group is
doubly transitive.  In that case, however,
known elementary combinatorial bounds
show that the order of the group is
quasipolynomially bounded.

\begin{theorem}[\cite{inventiones, pyber-doubly}] \label{thm:doubly}
Let $G\le S_n$ be a doubly transitive group
and assume $G\ngeq A_n$.  Then
$|G| \le \exp(O(\log n)^4)$.
\end{theorem}
\noindent
(The result in \cite{inventiones} gives the weaker upper
bound $\exp\exp(O(\sqrt{\log n}))$, which would also suffice
in our context since this quantity is less than $\exp(n^\epsilon)$
for all $\epsilon > 0$ and all sufficiently large $n$, so
it is much smaller than the threshold in Prop.~\ref{prop:cnsbound}.
The improved bound stated above was obtained 
in~\cite{pyber-doubly} using the framework of \cite{inventiones}.)
This completes the list of ingredients of an elementary
proof of Theorem~\ref{thm:saxl}.
\hfill $\Box$

\medskip
While a lemma in \cite{saxl} shows that the upper bound
in Theorem~\ref{thm:doubly} on the order of doubly transitive
groups other than $A_n$ and $S_n$ 
implies a nearly linear lower bound on the minimal
degree of these groups, we should mention that a
stronger, linear lower bound has been known for more
than 120 years.  The following result was proved by
Alfred Bochert in the 19th century by a lovely
combinatorial argument~\cite{bochert}.
\begin{theorem}[Bochert, 1897]
Let $G\le S_n$ be a doubly transitive group
and assume $G\ngeq A_n$.  Then the
minimal degree of $G$ is
$\mu(G) \ge n/8$.  For $n > 216$, the
lower bound improves to $n/4$.
\end{theorem}

\begin{remark}
Cameron's results~\cite{cameron81} classify all primitive permutation
groups of order greater than $n^{\log\log n}$ and naturally include
Theorem~\ref{thm:sunwilmes}.  The point here is that the proof by
Sun and Wilmes is elementary: it does not use the CFSG; in fact, it
uses no group theory at all.
\end{remark}

\subsection{Combinatorial relaxation of symmetry: Coherent configurations}
\label{sec:PCC}
The significance of the work of Sun and Wilmes goes far beyond
giving elementary proofs of group theoretic results previously
only known through the CFSG.  Their result is purely combinatorial;
it concerns primitive coherent configurations (PCCs): highly regular
colorings of the directed complete graph, with no symmetry assumptions.
We define this very general class of objects now.

\begin{definition}
A \emph{coherent configuration} (CC) is a pair $\mfX=(\Omega,c)$
where $\Omega$ is a set (the set of \emph{vertices}) and
$c:\Omega\times\Omega\to\Sigma$ is a coloring of the
ordered pairs of vertices ($\Sigma$ is the set of colors),
subject to the following regularity constraints.
We assume $c$ is surjective.  Below, $x,y,u,v\in\Omega$.
\begin{itemize}
\item[(i)] $(\forall x,y,z)(\text{if }c(x,x)=c(y,z)\text{ then }y=z)$.
\item[(ii)] $(\forall x,y,u,v)(\text{ if }c(x,y)=c(u,v)\text{ then }
        c(y,x)=c(v,u))$.
\item[(iii)] There exists a family of $|\Sigma|^3$ non-negative
integers $p_{i,j}^k$, called the \emph{intersection numbers},
such that
$$(\forall x,y)(\text{if }c(x,y)=k\text{ then }
    |\{z\ :\ c(x,z)=i,\ c(y,z)=j\}|=p_{i,j}^k)\,.$$
\end{itemize}
The number of colors used is called the \emph{rank} of $\mfX$.
\end{definition}

Let $G\le \sym(\Omega)$ be a permutation group.  Let $E_1,\dots,E_r$
denote the \emph{orbitals} of $G$, \ie, the orbits of the $G$-action on
$\Omega\times \Omega$.  Assigning color $i$ to the elements of $E_i$
we obtain a coloring $c:\Omega\times\Omega\to [r]$.  It is easy to see
that the resulting pair $\mfX(G):=(\Omega,c)$ is a CC,
and $G\le\aut(\mfX(G))$.  A CC arising in this manner is called
a \emph{Schurian CC}, after Issai Schur who first introduced CCs
in 1933, as a tool 
in the study of permutation groups~\cite{schur}.  CCs were
subsequently rediscovered several times in different contexts.
They include such much-studied structures as strongly regular
graphs, distance-regular graphs, association schemes.
If $X=(V,E)$ is a graph then the Weisfeiler--Leman color refinement
process~\cite{weisfeiler-leman, weisfeiler-book} (see, \eg,~\cite{quasipoly})
efficiently constructs a CC $\mfX(X)=(V,c)$ such that
$\aut(X)=\aut(\mfX(X))$.  In particular, $X$ has an asymmetric
$d$-coloring (of the vertices, as always in this paper) if and only if
$\mfX(X)$ has an asymmetric $d$-coloring.  CCs are also
critical ingredients in the recent isomorphism test~\cite{quasipoly}.
That paper includes a detailed introduction to the combinatorial
theory of CCs.

\begin{definition}[Constituents, PCC]
Given a CC $\mfX=(\Omega,c)$,
the digraphs $R_i=(\Omega,c^{-1}(i))$ $(i\in\Sigma)$ are the
\emph{constituent digraphs} of $\mfX$.  If one of these
is the diagonal $\diag(\Omega):=\{(x,x)\mid x\in\Omega\}$,
we call $\mfX$ \emph{homogeneous}.  Observe that a group
$G$ is transitive if and only if the corresponding Schurian
CC $\mfX(G)$ is homogeneous.  We call $\mfX=(\Omega,c)$
a \emph{primitive CC (PCC)} if it is homogeneous and
every non-diagonal constituent is a (strongly) connected
digraph.
\end{definition}

It is not difficult to show that a permutation
group $G$ is primitive if and only if $\mfX(G)$ is a PCC.
We should emphasize that conjecturally and empirically,
most CCs are not Schurian.

Note that for every $n$ there is essentially only one
rank-2 CC, namely, $\mfX(S_n)$, to which we refer as the
$n$-clique, and also as the trivial CC.

\begin{definition}[UPCC]
A \emph{uniprimitive coherent configuration (UPCC)}
is a PCC of rank $\ge 3$ (a nontrivial PCC).
\end{definition}

Now we can state the actual result of Sun and Wilmes.

\begin{theorem}[Sun--Wilmes]   \label{thm:sunwilmes2}
Let $\mfX$ be a UPCC with $n$ vertices.  If $n$ is sufficiently
large then $|\aut(\mfX)|\le \exp(O(n^{1/3}(\log n)^{7/3}))$,
unless $\mfX$ is the CC corresponding to a triangular graph
or a lattice graph.  
\end{theorem}

Finally, we are in the position to address Question~($\beta$)
above, by generalizing Theorem~\ref{thm:saxl} to UPCCs.

\begin{theorem}   \label{thm:uni-asymm}
All sufficiently large UPCCs admit an asymmetric 2-coloring.
\end{theorem}

Like much of the literature about asymmetric colorings, 
we shall rely on the following lemma,
implicit in the counting argument used by
Gluck~\cite{gluck} and Cameron et al.~\cite{saxl} and
made explicit by Russell and Sundaram a decade and a
half later~\cite{russell}.

\begin{prop}[Motion Lemma, \cite{gluck}, \cite{saxl}, \cite{russell}]
\label{prop:motionlemma}
Let $G\le S_n$ be a permutation group of minimal degree $\mu$.
If $d^{\mu/2} \ge |G|$ then $G$ admits an asymmetric $d$-coloring.
\hfill $\Box$
\end{prop}

In order to take advantage of this lemma, we need
and upper bound on the order of the automorphism
group of $\mfX$, and a lower bound on the motion of $\mfX$.

The former is provided by the Sun--Wilmes result,
Theorem~\ref{thm:sunwilmes2}.

We take the latter from 
a 1981 paper of this author~\cite{annals}.
Since the result is not explicitly stated in~\cite{annals},
let me show how it follows immediately from the
main technical result of that paper.  (This fact
has been known since immediately after the
publication of~\cite{annals}.)

Let $\mfX=(\Omega,c)$ be a UPCC.
Following~\cite{annals}, we say that vertex $z$
\emph{distinguishes} vertices $x$ and $y$ if $c(z,x)\neq c(z,y)$.
Let $D(x,y)$ denote the set of vertices $z$ that distinguish
$x$ and $y$.  The core technical result of~\cite{annals} is the
following.

\begin{theorem}[\cite{annals}]   \label{thm:annals}
Let $\mfX$ be a UPCC with $n$ vertices.  Then, for every pair
$x,y$ of distinct vertices, $|D(x,y)|\ge (\sqrt{n}-1)/2$.
\end{theorem}
All we need to add to this result is the following
observation.

\begin{obs}   \label{obs:uni-motion}
Let $\mfX$ be a UPCC with $n$ vertices.  Then the motion
of $\mfX$ is $\ge \min_{x\neq y} |D(x,y)|$.
\end{obs}
\begin{proof}
Let $T=\supp(\sigma)$ for some $\sigma\in\aut(X)$
such that the size of $T$ is the motion of $\mfX$.
Let $x\in T$ and $y=\sigma(x)$.  Then $y\neq x$ by
definition.  We claim that $T\supseteq D(x,y)$.
Indeed, if $z\in\Omega\setminus T$ then
$c(z,x)=c(\sigma(z),\sigma(x))=c(z,y)$, so $z\not\in D(x,y)$.
\end{proof}

\begin{proof}[Proof of Theorem~\ref{thm:uni-asymm}]
Let $\mu=\mu(\mfX)$ be the motion of $\mfX$.
We have $\mu \ge (\sqrt{n}-1)/4$ by Theorem~\ref{thm:annals} 
and Obs.~\ref{obs:uni-motion}.  So
$2^{\mu/2} \ge 2^{(\sqrt{n}-1)}/4$.
For sufficiently large $n$, this quantity is greater than the
Sun--Wilmes bound (Theorem~\ref{thm:sunwilmes2}),
which is $\exp(C(n^{1/3}(\log n)^{7/3}))$ for some constant $C$.  
\end{proof}

\section{Open problems}
\label{sec:open}
Theorem~\ref{thm:finite} describes a finite version of our main
technical result, Theorem~\ref{thm:main1}.  I would be most
interested in more effective versions of this result, and
specifically in Conjectures~\ref{conj:finite1} and~\ref{conj:finite2}
(Polylog bound conjecture).

Below I list a number of additional problems and directions
of study.
All groups in this section, except in Problems (1), (2), and (11),
are finite.

\mn
{\bf Terminology.}\quad  Given a permutation group $G\le\sym(\Omega)$,
recall that we say that a coloring $\gamma$ ``results in a 2-group''
if $\gamma$ is a coloring of the permutation domain $\Omega$
and $G_\gamma$ (the color-preserving subgroup of $G$)
is a 2-group.  And we can substitute any class of groups
in such a statement for ``2-groups,'' so for instance
the statement that ``a coloring results in a group with 
derived length $\le 3$'' should have a clear meaning.
PS-closed classes of groups (classes closed under direct
products and subgroups), such as those mentioned above,
are of particular interest because of their monotonicity
properties described in Obs.~\ref{obs:ps-closed} and
Remark~\ref{rem:ps-closed}.
\begin{enumerate}
\item    
\begin{itemize}
\item[(a)]
  Give a CFSG-free proof of Theorem~\ref{thm:main1}
  and thereby to the Infinite Motion Conjecture.\\
  I expect that the results mentioned in the preceding section,
  and in particular the Sun--Wilmes Theorem
  (Theorem~\ref{thm:sunwilmes}), will play a role.

\item[(b)] How much of Theorem~\ref{thm:simple}
  (``Reducing simple image'') can be salvaged without CFSG?
  \end{itemize}
    
\item  
\begin{itemize}  
   \item[(a)]   
  Does there exist a constant $C$ such that the following
  strengthening of the Infinite Motion Conjecture holds?\\
  \emph{Let $X$ be a connected locally finite rooted graph with
  infinite motion.  Then $X$ has an asymmetric 2-coloring (red/blue)
  that is overwhelmingly blue in the sense that
  every sphere about the root gets at most $C$ red vertices.}\\
  This question is motivated by the consideration of the ``cost''
  of coloring, as defined below.  It is easy to see
  that the statement is true for locally finite rooted trees 
  without vertices of degree~1.

   \item[(b)] 
  Does there exist a constant $C$ such that the following
  strengthening of Theorem~\ref{thm:simple}
  (``Reducing simple image'') holds?\\
  \emph{Let $G\le \sym(\Omega)$, where $\Omega$ is a finite set.
    Let $\vf: G\onto T$
    be an epimorphism where $T$ is a nonabelian simple group.  Then 
    there exists a subset $\Delta \subseteq\Omega$ of size
    $|\Delta|\le C$ such that $\vf(G_\Delta) < T$.}\\
  We note that $C=1$ will not suffice, as the example
  $G=\zzz_p^d \semidirect T\le \agl(d,p)$ shows, where the
  semidirect product is defined by a nontrivial $d$-dimensional
  irreducible representation of $T$ over $\fff_p$.
  \end{itemize}  

\item   
  Recall Theorem~\ref{thm:nonsolvable} (reducing non-solvable image):
\vspace{0.1cm}

  \noindent
  \emph{Let $G\le \sym(\Omega)$ where $\Omega$ is a finite set.
  Let $H$ be a group and $\vf: G\onto H$ an epimorphism.
  Then $(\exists \Delta \subseteq\Omega)(\vf(G_\Delta) < H)$,
  assuming $H$ is not solvable.}

\vspace{0.1cm}
  \begin{itemize}    
  \item[(a)]  We note that the condition that ``$H$ is not solvable''
  cannot be replaced by the condition ``$|H| > 1$,'' as shown
  by the example $G=D_k$ (the dihedral group of order $2k$ acting
  naturally on $k$ elements) and $H=\zzz_2$ where the epimorphism
  $\vf$ is defined by the natural epimorphism $D_k \onto D_k/\zzz_k$,
  where $3\le k\le 5$.
  \item [(b)]  Question.  Does there exist a number $C$ such that
  the following holds?\\
  \emph{Let $G\le \sym(\Omega)$ where $\Omega$ is a finite set.
  Let $H$ be a group.    Let $\vf: G\onto H$ be an epimorphism.
  Assume $|H|\ge C$.  
  Then $(\exists \Delta \subseteq\Omega)(\vf(G_\Delta) < H)$.}
\item[(c)]  Does the conclusion of (3)(b) follow  if we only require
  $|G| \ge C$ and $|H|\ge 2$ ?
  \end{itemize}  

\item    
  \begin{itemize}  
  \item[(a)]
Given a sequence $n_0,\dots,n_k$ of positive integers,
consider the \emph{balanced rooted tree} of height $k$
where the vertices at distance $j$ from the root
have $n_j$ children.  So the automorphism group is
the wreath product of the symmetric groups of degree $n_j$.
What is the asymmetric coloring number of these trees?

  \item[(b)]  More generally,
how does the \emph{position of symmetric and
alternating groups in a structure tree} (hierarchy of
blocks of imprimitivity) of a transitive group
affect the asymmetric coloring number? 
  \end{itemize}

\item  
A systematic study of \emph{solvable colorings}
for permutation groups would be of interest. Recall that
these are colorings of the permutation domain that 
result in a solvable group.
More specific questions on this subject follow.

  \begin{itemize}
  \item[(a)]
Within various classes of permutation groups, characterize those
that do not admit a solvable 2-coloring.\\
Among primitive groups, the only groups that do not admit
a solvable 2-coloring are the symmetric and alternating
groups of degree $\ge 9$ in their natural action.

  \item[(b)]
The wreath product $S_8\wr S_5$ does not admit a solvable
2-coloring.  Let us now consider the \emph{transitive permutation
groups without alternating composition factors}.
Can we characterize, which of these do not admit a solvable
2-coloring?

  \item[(c)]
The automorphism group of every tournament is solvable.
This statement is equivalent to the Feit--Thompson Theorem.
Can we prove without using heavy group theory that
\emph{all tournaments have a solvable $k$-coloring
for some fixed value $k$?}  Or is such a statement still
equivalent to Feit--Thompson?  
  \end{itemize}  

\item  
A general theme is, what kind of structural
reductions of the group can be achieved by a bounded number
of colors.  Here is a specific question of this type.

  \begin{itemize}  
  \item[(a)]
Does there exist a number $g_0$ such that the following
holds: \emph{Every permutation group admits a 2-coloring that 
kills all non-alternating composition factors of order 
$\ge g_0$}, \ie, after the coloring, all composition 
factors will either be alternating or of order $ < g_0$.
  \end{itemize}

\item  
Some questions of this type arose in this
paper when starting from a solvable group.

  \begin{itemize}  
  \item[(a)] What is the smallest $c$ such that 
\emph{every solvable permutation group admits a
2-coloring that results in derived length $\le c$ ?}
Such a $c$ exists by Theorem~\ref{thm:bded-orb}.

  \item[(b)] What is the smallest $C$ such that
\emph{every solvable permutation group admits a 
2-coloring that reduces the length of all orbits to $\le C$ ?}
Such a $C$ exists by Theorem~\ref{thm:bded-orb}.
The group $S_4\wr S_4$ shows that $C$ cannot be less than 4.

  \item[(c)] Does every solvable group have a 3-coloring
that results in a 2-group?\\
Two colors do not suffice for this, as the
example of $S_4\wr K$ shows for any solvable group $K$
that is not a 2-group.

  \item[(d)] Does every solvable permutation group have a 2-coloring  
that results in an \emph{abelian-by-2-group}, \ie, in
a group that has an abelian normal subgroup with the 
quotient being a 2-group?

  \item[(e)]
Sometimes instead of solvability of the automorphism group,
we can assume something about the underlying structure.
We discussed tournaments above.  Another example:
If $X$ is a connected cubic graph
then it has a \emph{low-cost} (see below) 2-coloring
that results in a 2-group: just color a pair of
adjacent vertices red, the rest blue.

For the same reason, if $X$ is a connected graph
such that every vertex has degree $\le k$ then
the same low-cost 2-coloring results in a group
with bounded composition factors.  (Every composition
factor is a subgroup of $S_{k-1}$.)  This fact
was used by Gene Luks to revolutionize the
theory of Graph Isomorphism testing in
1980~\cite{luks-bded}.
  \end{itemize}
  
\item  
An important direction of study is the extension of
known results about the minimal degree of primitive groups 
(often obtained via CFSG)
to the motion of \emph{strongly regular graphs,}
\emph{distance-regular graphs,} and \emph{primitive coherent
configurations (PCCs)}.  Some work in this direction has already been
done, see, \eg, \cite{annals, itcs14}, and the profound results
in \cite{sunwilmes, kivva21a, kivva21b, kivva-rank4}. A conjecture of this
author that motivates Kivva's work~\cite{kivva21a,kivva21b,kivva-rank4}
is the following.
\begin{conjecture}
Let $\mfX$ be a PCCs
with $n$ vertices.  If $\mfX$ is not a Cameron scheme
then the motion of $\mfX$ is $\ge cn$
for some positive constant $c$.
\end{conjecture}
The conjecture is motivated by a
1991 result by Liebeck and Saxl~\cite{liebecksaxl-mindeg}
that says that the statement is true with $c=1/3$ 
in the Schurian case.

The Cameron schemes are Schurian.  They correspond to
primitive permutation groups $G$ acting on $n=\binom{m}{k}^r$
elements for some $m\ge 5$, $1\le k \le m-1$, and $r\ge 1$,
as a subgroup of $S_m\wr S_r$ containing $A_m^r$
where $S_m$ acts on the $k$-subsets of an $m$-set
and $S_r$ acts on the ordered $r$-tuples of such subsets.

In his monumental work,
Kivva \emph{confirmed the conjecture for PCCs of rank $\le 4$}
\cite{kivva-rank4} 
and for \emph{distance-regular graphs of bounded diameter}
\cite{kivva21a,kivva21b}.
Given the combinatorial nature of the question,
no group theory is involved in his proofs.

\item 
\emph{Combinatorial symmetry breaking is a key aspect
of the Graph Isomorphism problem} (see~\cite{quasipoly}).
From this perspective,
the cost of breaking the symmetry is not the number of
colors used but the \emph{entropy of the distribution
of colors:} if color $i$ occurs $k_i$ times on a set of size 
$n=\sum_{i=1}^s k_i$ then we are looking at the quantity
$H(k_1/n,\dots,k_s/n)=-\sum (k_i/n)\log_2 (k_i/n)$.

If the $s$ colors are uniformly distributed ($k_i=n/s$ for all $i$) then
the entropy is $\log_2 s$.  On the other hand, if one color dominates
and all the orther colors occupy just a small portion of the domain
then the entropy is close to zero.  For instance, in the case
of a 2-coloring, which is equivalent to fixing a subset, we wish
that subset to be as small as possible.  The size of that
set as a cost measure was introduced by Debra Boutin in
2008~\cite{boutin} (see also \cite{boutinimrich}) from
the philosophical consideration that,
given that in most cases of interest, an asymmetric 2-coloring
exists, a more refined measure of the cost of symmetry breaking is
needed.  This measure of cost represents a \emph{convergence with the
classical concept of minimum bases of permutation groups.}

A \emph{base} $\Delta\subseteq\Omega$ is a subset such that the
pointwise stabilizer is trivial: $G_{(\Delta)}=1$.
Bases have been introduced in computational group
theory (Sims~\cite{sims}) in the 1960s with the express purpose
of breaking all symmetry, but the concept also has great
theoretical significance (see below).

A base of size $b$ gives a coloring with $b+1$ colors:
individual colors for each element of the base, and a single color
for the rest.  Bases provide the \emph{lowest entropy} among all colorings
with $b+1$ colors.  But bases may not use the optimal number
of colors for the type of questions we are considering here; for 
instance, the base size for $S_8$ is 8, but
$S_8$ has a solvable 2-coloring.

In any case, the sizable literature about minimal bases of 
permutation groups will be particularly relevant in the
context of the refined cost measures.  Here is a selection
of relatively recent papers on minimum bases:
\cite{liebeck84, pyber,seress97,
gluckseress, liebeckshalev, burness11, burness17}.
One of the recent motivators of the area has been \emph{Pyber's
base size conjecture} (1993)~\cite[p. 207]{pyber}, resolved
in~\cite{duyhan} (2018) and made effective in \cite{halasi} (2019).
The latter paper also includes a nice overview of the subject.

\item 
  \label{pr:cameron-schemes}
Symmetry-breaking by coloring of primitive coherent configurations
is at the heart of the study of the \emph{Graph Isomorphism}
problem~\cite{quasipoly}.  One of the types of problems that
arise there is to distinguish the Johnson, Hamming, and
Cameron schemes from all other primitive coherent configurations
by showing that for all other configurations, 
symmetries are destroyed at much lesser cost (see~\cite{quasipoly}).
Here the ``cost'' refers to the refined cost measures
explained in the previous item.

\item 
  \label{finite-subdegree}
Let me highlight an interesting generalization of Tucker's
Conjecture, proposed by Imrich, Smith, Tucker, and Watkins,
that states that
\emph{a closed permutation group $G$, acting on a countably 
infinite set, with infinite minimal degree and finite subdegrees,
admits an asymmetric 2-coloring} (``Infinite Motion Conjecture
for Permutation Groups'') \cite[Sec 4]{watkins15}.  Here ``closed''
means a closed subgroup of the symmetric group in the permutation
topology, where a neighborhood basis of the identity consists of 
the pointwise stabilizers of finite subsets of the permutation domain.
The subdegrees are the lengths of the orbits of the stabilizer of a point.
\end{enumerate}

Let me conclude with a conjecture I have been entertaining for decades.
We have seen that epimorphisms to the alternating groups give
us a lot of trouble.  (Such epimorphisms have also defined
the bottleneck for Luks's graph isomorphism test that
caused three decades without progress on that problem,
see \cite{quasipoly}.)
I believe the conjecture below relates
to our subject, although I cannot draw a formal connection.

Answering this author's question, in 1983, 
Martin Liebeck~\cite{liebeck83} proved
that if $X$ is a graph and $\aut(X)\cong A_k$ then $n$ (the
number of vertices) must grow exponentially as a function
of $k$.  Specifically, he showed 
that $n\ge 2^k-k-2$ for all $k\ge 13$ and that
this lower bound is tight for $k\equiv 0$ or $1\pmod 4$.

\begin{conjecture}
There exists a constant $C > 1$ such that the following holds.
Let $X$ be a graph with $n$ vertices.  Assume $\aut(X)$ has
an epimorphism onto the alternating group $A_k$ $(k\ge 3)$.
Then $n \ge C^k$.
\end{conjecture}

\section*{Acknowledgments}
First and foremost, I would like to thank 
Wilfried Imrich, my friend of over half a century,
for bringing the Tucker conjecture to my attention
and also for his insistence that I attend
the BIRS/CMO workshop on 
``Symmetry Breaking in Discrete Structures.''

I wish to thank the BIRS-affiliated
\emph{Casa Matem\'atica Oaxaca (CMO)} for their hospitality
during the workshop, September 16--21, 2018.

During the workshop I had the opportunity
to learn about current research in this area,
including work on Tom Tucker's conjecture addressed in this paper.
It is a pleasure to acknowledge productive conversations
with Wilfried Imrich, Florian Lehner, Monika Pil\'sniak,
Tom Tucker, and other participants of the meeting.

My special thanks are due to my recent student John Wilmes
and my current student Bohdan Kivva, for all the insights
they shared with me over the years, some of which
turned out to be relevant to this paper.

I thank Saveliy Skresanov~\cite{skresanov}
for pointing out Choi's work~\cite{choi}
and his own GAP search
(see Prop.~\ref{prop:mathieu} and Remark~\ref{rem:mathieu}).

Last but not least, I'd like to pay tribute to Jan Saxl,
whose lifelong influence on my work started in 1979
when his article with Cheryl Praeger~\cite{praeger}
inspired my entry into the theory of primitive
permutation groups~\cite{annals}, and continued
until recently with work on our joint paper~\cite{bps}.
His imprint is discernible throughout this article.

\medskip
The research presented in this paper was supported in part
by NSF Grant CCF-1718902.  The views expressed in the paper 
are solely the author's and have not been evaluated or
endorsed by the NSF.

\end{document}